\documentclass[12pt]{amsart}
\usepackage{amscd,amssymb,amsthm,verbatim}
\newcommand{\C}{{\mathbb C}}

\newcommand{\const}{\operatorname{const.}}

\newcommand{\diam}{\operatorname{diam}}

\newcommand{\Dom}{\operatorname{Dom}}

\newcommand{\Hess}{\operatorname{Hess}}

\newcommand{\inj}{\operatorname{inj}}

\newcommand{\R}{{\mathbb R}}

\newcommand{\Ric}{\operatorname{Ric}}

\newcommand{\Rm}{\operatorname{Rm}}

\newcommand{\spann}{\operatorname{span}}

\newcommand{\vol}{\operatorname{vol}}
\newcommand{\Z}{{\mathbb Z}}

\numberwithin{equation}{section}
\theoremstyle{plain}
\newtheorem{definition}[equation]{Definition}

\newtheorem{lemma}[equation]{Lemma}
\newtheorem{theorem}[equation]{Theorem}
\newtheorem{proposition}[equation]{Proposition}
\newtheorem{corollary}[equation]{Corollary}

\theoremstyle{remark}

\newtheorem{remark}[equation]{Remark}

\usepackage{graphicx}
\input{amssym.def}
\input{amssym}
\input{epsf}
\usepackage{a4wide}

\begin{document}

\title[Comparison geometry]
      {Comparison geometry of holomorphic bisectional curvature
        for K\"ahler manifolds and limit spaces}

\author{John Lott}
\address{Department of Mathematics\\
University of California, Berkeley\\
Berkeley, CA  94720-3840\\
USA} \email{lott@berkeley.edu}

\thanks{Research partially supported by NSF grant
DMS-1810700}
\date{January 2, 2021}
\begin{abstract}
  We give an analog of triangle comparison for K\"ahler manifolds with a lower
  bound on the holomorphic bisectional curvature.  We show that the
  condition passes to noncollapsed Gromov-Hausdorff limits. We discuss
  tangent cones and singular K\"ahler spaces.
\end{abstract}

\maketitle

\section{Introduction} \label{sect1}

Holomorphic bisectional curvature is a K\"ahler analog of Riemannian
sectional curvature.  We recall the definition in Section \ref{sect3}.
There is a well developed theory of Riemannian manifolds with lower
sectional curvature bounds, including such topics as triangle
comparison, Gromov-Hausdorff limits and Alexandrov spaces.  The
goal of this paper is to give K\"ahler analogs.

To state the first main result, we define a modified distance-squared
function.
Given $d \ge 0$ and $K \in \R$, define $d_K \ge 0$ by
\begin{equation} \label{1.1}
  d_K^2  =
  \begin{cases}
    - \: \frac{4}{K} \log \cos \left( d \sqrt{\frac{K}{2}} \right) &
    \mbox{ if } K > 0, \\
    d^2 & \mbox{ if } K = 0, \\
    \frac{4}{-K} \log \cosh \left( d \sqrt{\frac{-K}{2}} \right) &
    \mbox{ if } K < 0.
\end{cases}
\end{equation}
(If $K > 0$ then we restrict to $d \le \frac{\pi}{\sqrt{2K}}$.)
Let $M$ be a complete K\"ahler manifold.
Given $p \in M$ and $K \in \R$,  let $d_p \in C(M)$ be the
distance from $p$ and define $d_{K,p}$ using (\ref{1.1}), replacing
the $d$ in the
right-hand side by $d_p$.

We write $BK \ge K$ if the holomorphic bisectional curvatures of $M$ are
bounded below by $K \in \R$. We prove the following analog of
triangle comparison.

\begin{theorem} \label{1.2}
  Let $M$ be a complete
  K\"ahler manifold. Given $K \in \R$, the manifold $M$ has $BK \ge K$
  if and only if it satisfies the following property.
  Let $i \: : \: \overline{D^2} \rightarrow M$ be an embedding of a disk
  into $M$, that is holomorphic on $D^2$. Let $\Sigma$ be the
  image of $i$. 
  Let $dA$ denote the area form on $\Sigma$.
  Let $z$ be the local coordinate on $D^2$ and let
  $\theta \in [0, 2 \pi)$ be the local coordinate on
    $\partial{\overline{D^2}}$.
  Then
  \begin{equation} \label{1.3}
    d^2_{K,p}(0) \ge \frac{2}{\pi} \iint_\Sigma \log |z| \: dA +
    \frac{1}{2\pi} \int_{\partial \Sigma} d^2_{K,p}(\theta) d\theta,
  \end{equation}
  where the
  ``$0$'' on the left-hand side denotes $i(0)$, the center of $\Sigma$.
  \end{theorem}

Next, we consider noncollapsing sequences of complete pointed
K\"ahler manifolds with $BK \ge K$. Lee and Tam showed that after passing
to a subsequence, there is a pointed Gromov-Hausdorff limit that is a
complex manifold
\cite{Lee-Tam (2019)}.  Regarding its geometry,
we show that (\ref{1.3}) holds on the limit.

\begin{theorem}  \label{1.4}
Let $\{(M_i, p_i, g_i)\}_{i=1}^\infty$ be a sequence of pointed
$n$-dimensional complete K\"ahler manifolds with
$BK \ge K$.
Suppose that there is some
$v_0 > 0$ so that for all $i$, we have $\vol(B(p_i, 1)) \ge v_0$. Then after
passing to a subsequence, there is a pointed Gromov-Hausdorff limit
$(X_\infty, p_\infty, d_\infty)$ with the following properties.
\begin{enumerate}
\item $X_\infty$ is a complex manifold.
\item Embedded holomorphic disks $\Sigma$ in $X_\infty$ satisfy (\ref{1.3}), where
$dA$ is now the two dimensional Hausdorff measure coming from $d_\infty$.
 \end{enumerate}
\end{theorem}

Some simple examples of such limit spaces come from two dimensional length spaces
with Alexandrov curvature bounded below.
The proof of Theorem \ref{1.4} uses local Ricci flow techniques, as developed by
Bamler-Cabezas-Rivas-Wilking \cite{Bamler-Cabezas-Rivas-Wilking (2019)},
Hochard \cite{Hochard (2019)},
Lee-Tam \cite{Lee-Tam (2017)} and Simon-Topping \cite{Simon-Topping (2017)}.

The content of the paper is as follows.  In Section \ref{sect2} we briefly
recall some facts about Riemannian manifolds with nonnegative
sectional curvature, and their Gromov-Hausdorff limits. In
Section \ref{sect3} we show
\begin{itemize}
\item A complete K\"ahler manifold has $BK \ge K$ if and only if
  $\sqrt{-1} \partial \overline{\partial} d_{K,p}^2/2 \le \omega$ as
  currents.
\item Theorem \ref{1.2} holds.
\item If a Hermitian manifold satisfies (\ref{1.3}) then it must be
  K\"ahler.
\item A domain $M$ in a model space (of constant holomorphic sectional
  curvature)
  satisfies (\ref{1.3}) if and only if the length metric on $M$ is the same
  as the restricted metric from the model space.
\end{itemize}

Section \ref{sect4} is about noncollapsed
pointed Gromov-Hausdorff limits.  We prove 
Theorem \ref{1.4} and construct local K\"ahler potentials
$\{\phi_\alpha\}$ on the limit space.

In Section \ref{sect5} we give a notion of ``$BK \ge K$'' for
possibly singular complex spaces.  We use the notion of
K\"ahler spaces from \cite{Moishezon (1975)}, which is formulated
in terms of local potential functions
$\{\phi_\alpha\}$.  We define metric K\"ahler
spaces and an associated complex Gromov-Hausdorff convergence, which may
be of independent interest.  We say that a
metric K\"ahler space has 
``$BK \ge K$'' if $\phi_\alpha - d_{K,p}^2/2$ is
plurisubharmonic for all $\alpha$ and $p$.  For normal complex spaces,
this is equivalent to (\ref{1.3}) being satisfied. The following
properties hold:
\begin{itemize}
\item
  Given a sequence of metric K\"ahler spaces with ``$BK \ge K$'', if it
  converges in the pointed complex Gromov-Hausdorff sense then the
  limit space has ``$BK \ge K$''.
\item Under the assumptions of Theorem \ref{1.4}, a subsequence converges in
  the pointed complex Gromov-Hausdorff sense.
\item If a K\"ahler orbifold has $BK \ge K$ in the sense of curvature
  tensors then its underlying length space has ``$BK \ge K$''.
\end{itemize}

  Section \ref{sect6} is about tangent cones of the limit spaces from Theorem \ref{1.4}.
  We show
  \begin{itemize}
  \item A tangent cone is a K\"ahler cone that is biholomorphic to
    $\C^n$.
  \item When the distance function from the vertex is radially homogeneous on
    $\C^n$, the tangent cone is an affine cone over a copy of 
    $\C P^{n-1}$ with
     ``$BK \ge 2$'', in the sense of the previous section.
  \end{itemize}

I thank Man-Chun Lee, Gang Liu and Song Sun for helpful comments.
I also thank Man-Chun for pointing out a gap in an earlier version of
the paper, and the referee for useful remarks.

\section{Some facts from Riemannian comparison geometry} \label{sect2}

Let $(M,g)$ be a complete Riemannian manifold.  We consider lower
sectional curvature bounds; for simplicity,
we assume that $(M,g)$ has
nonnegative sectional curvature.
Given $p \in M$, let $d_p \in C(M)$ denote the Riemannian distance from
$p$. Then
\begin{equation} \label{2.1}
  \Hess(d_p^2/2) \le g
\end{equation}
away from the cut locus $C_p$ of $p$.

Let $\{\gamma(t)\}_{t \in [0,L]}$
be a unit-speed geodesic
in $M-C_p$. For brevity, we write $d_p(t)$ for $d_p(\gamma(t))$.
It follows from (\ref{2.1}) that
$\frac{d^2}{dt^2} \left( d_p^2(t)/2 \right) \le 1$, i.e.
$\frac{d^2}{dt^2} \left( d_p^2(t)/2 - t^2/2 \right) \le 0$. That is,
$d_p^2(t) - t^2$ is concave on $[0,L]$. Then
\begin{equation} \label{2.2}
  d_p^2(t) - t^2 \ge \frac{t}{L} (d_p^2(L) - L^2 ) +
  \left( 1-\frac{t}{L} \right) d_p^2(0),
\end{equation}
or
\begin{equation} \label{2.3}
    d_p^2(t) \ge \frac{t}{L} d_p^2(L) +
  \left( 1-\frac{t}{L} \right) d_p^2(0) - t(L-t).
  \end{equation}
Toponogov's theorem says that (\ref{2.3}) remains true without the restriction
that $\gamma$ lies in $M-C_p$.

\begin{remark}  \label{2.4} 
We state some facts without proof.
\begin{enumerate} 
\item Equation (\ref{2.3}), when applied to minimizing geodesics, passes to
  pointed Gromov-Hausdorff limits.  That is, such a limit is a complete
  length space with nonnegative Alexandrov curvature.
\item A noncollapsed limit is a topological manifold \cite{Perelman (1991)}.
\item A tangent cone of a noncollapsed limit is a metric cone.  Its link
  has Alexandrov curvature bounded below by one
  \cite[Corollary 7.10]{Burago-Gromov-Perelman (1992)} and is
  homeomorphic to a sphere \cite[Theorem 1.3]{Kapovitch (2002)}.
\item A Finsler manifold with nonnegative Alexandrov curvature is a
  Riemannian manifold.
\item A polytope in Euclidean space, i.e. a connected
  finite union of top dimensional
  simplices, has nonnegative
  Alexandrov curvature, with respect to the length metric,
  if and only if it is convex.
\end{enumerate}
\end{remark}

\section{Comparison geometry for K\"ahler manifolds with lower bounds on
holomorphic bisectional curvature} \label{sect3}

\subsection{Holomorphic bisectional curvature}

Let $M$ be an $n$-dimensional
K\"ahler manifold.  We let $\omega$ denote its K\"ahler form.
In terms of holomorphic normal coordinates at a point $p$, we have
$\omega(p) = \frac{\sqrt{-1}}{2} \sum_{i=1}^n dz^i \wedge d\overline{z}^i$.

Suppose that $n \ge 2$.
Given $p \in M$, if $\sigma$ and $\sigma^\prime$ are $J$-invariant $2$-planes
(i.e. complex lines) in
$T_pM$, write $\sigma = \spann(X, JX)$ and
$\sigma^\prime = \spann(Y, JY)$ for unit vectors $X$ and $Y$. The
holomorphic bisectional curvature of $\sigma$ and $\sigma^\prime$
is $H(\sigma, \sigma^\prime) = R(X,JX,Y,JY)$.
If $\sigma = \sigma^\prime$ then the holomorphic sectional curvature
of $\sigma$ is $H(\sigma, \sigma)$.
From the Bianchi identity,
\begin{equation} \label{3.1}
  R(X,JX,Y,JY) = R(X,Y,X,Y) + R(X,JY,X,JY).
\end{equation}
In particular,
\begin{equation} \label{3.2}
  (\mbox{sect. curv. } \ge \const) \implies
  (\mbox{holo. bisec. curv. } \ge \const) \implies
    (\mbox{Ricci curv. } \ge \const) 
  \end{equation}
where the constants are related by $n$-dependent factors.
Given $K \in \R$, we say that $BK \ge K$ if all of the holomorphic bisectional
curvatures are bounded below by $K$.

We use the curvature notation of
\cite[Chapter 9]{Kobayashi-Nomizu (1969)}. In particular, if
$\{e_i, e_j\}$ are elements of a unitary frame then the corresponding
holomorphic bisectional curvature is $- R_{i \overline{i} j \overline{j}}$.
(Note the minus sign.)
Hence $BK \ge K$ if and only if we have
\begin{equation} \label{3.3}
  - R(X, \overline{X}, Y, \overline{Y}) \ge
  K \left( \langle X, \overline{X} \rangle \langle Y, \overline{Y} \rangle +
  \langle X, \overline{Y} \rangle \langle Y, \overline{X} \rangle
\right)
\end{equation}
for all $X, Y \in T^{(1,0)}M$.
(If $n = 1$ then to be consistent with (\ref{3.3}), we say that
$BK \ge K$ if the holomorphic sectional curvatures are bounded below
by $2K$.)

The metric on $\C P^n$ with constant holomorphic sectional curvature $c$ is
\begin{equation} \label{3.4}
  g_{i \overline{j}} =
  \frac{4}{c} \partial_i \overline{\partial}_j
\log \left( 1 + \frac{c}{4} |z|^2 \right)
\end{equation}
with curvature tensor
\begin{equation} \label{3.5}
  R_{i \overline{j} k \overline{l}} =
  - \: \frac{c}{2} \left( g_{i \overline{j}}
  g_{k \overline{l}} +
  g_{i \overline{l}}
  g_{k \overline{j}} \right).
\end{equation}
The Riemannian sectional curvatures lie in $\left[ \frac{c}{4}, c \right]$.
The holomorphic bisectional curvatures lie in $\left[ \frac{c}{2}, c \right]$.
The diameter is $\pi c^{- \: \frac12}$.
(If $n=1$ then the Riemannian sectional curvature and the holomorphic
bisectional curvature are $c$, and the diameter is 
$\pi c^{- \: \frac12}$.)

  If $BK \ge K > 0$ then $\diam(M) \le \frac{\pi}{\sqrt{2K}}$
  \cite{Li-Wang (2005)}. It seems to be open whether equality
  implies that $(M,g)$ is the Fubini-Study metric on $\C P^n$, up to a
  constant \cite{Liu-Yuan (2018),Tam-Yu (2012)}.

  A compact K\"ahler manifold with positive holomorphic
  bisectional curvature is
  biholomorphic to a complex projective space
  \cite{Mori (1979),Siu-Yau (1980)}. The nonnegative case was described in
  \cite{Mok (1988)}. Alternative proofs of these results, along with
  extensions to transverse Sasakian geometry, are in
  \cite{He-Sun (2016), He-Sun (2015)}.
  
\subsection{Differential inequality for smooth K\"ahler manifolds}

We now give a K\"ahler analog of (\ref{2.1}), for $BK \ge K$.

For $p \in M$, let $d_p$ denote the distance function from $p$ and define
$d_{K,p}$ using (\ref{1.1}), with $d$ replaced by $d_p$.

\begin{proposition} \label{3.6}
  Let $M$ be
  a complete K\"ahler manifold. If $BK \ge K$ then for all $p \in M$,
\begin{equation} \label{3.7}
  \sqrt{-1} \partial \overline{\partial} d_{K,p}^2/2 \le \omega
\end{equation}
as currents on $M$.
\end{proposition}
\begin{proof}
  Suppose that $BK \ge K$. 
  (If $K > 0$ then we initially restrict to the case when
  $\diam(M) < \frac{\pi}{\sqrt{2K}}$.)
  It follows from \cite[Theorem 2.1]{Tam-Yu (2012)}, along with some
  calculation, that (\ref{3.7}) is satisfied
  smoothly away from the cut locus of $p$. 
  Given $q \in M - \{p\}$, let $\phi$ be a local K\"ahler potential in a
  neighborhood $U$ of $q$, i.e.
  $\omega = \sqrt{-1} \partial \overline{\partial} \phi$. We can
  assume that $p \notin U$.
  To prove (\ref{3.7}),
  we wish to show that $\phi - d_{K,p}^2/2$ is plurisubharmonic.
  For this, it suffices to show that it is subharmonic on any embedded
  holomorphic disk $\Sigma$ in $U$, i.e. that
  $\triangle_{\Sigma} d_{K,p}^2 \le 4$ as measures on $\Sigma$.

  Given $m \in \Sigma$, we will construct a barrier function at $m$.
  Let $\gamma : [0, d(p,m)] \rightarrow M$ be a minimizing unit speed
  geodesic from $p$ to $m$. Let $F_K$ be the function appearing on
  the right-hand side of
  (\ref{1.1}), so $d_K^2 = F_K \circ d$. Then $F_K^\prime \ge 0$ and
  $F_K^{\prime \prime} \ge 0$. For small $\epsilon > 0$, consider
  $F_K \circ (d_{\gamma(\epsilon)} + \epsilon)$. Its value at
  $m$ is $d_{K,p}^2(m)$. As $d_{\gamma(\epsilon)} + \epsilon \ge d_p$,
  it follows that $F_K \circ (d_{\gamma(\epsilon)} + \epsilon) \ge
  d_{K,p}^2$.

  Since $m$ is not in the cut locus
  of $\gamma(\epsilon)$, we now know that
\begin{equation} \label{3.8}
  \triangle_{\Sigma} (F_K \circ d_{\gamma(\epsilon)}) \le 4
  \end{equation}
in a neighborhood of $m$ in $\Sigma$.
As
\begin{equation} \label{3.9}
  \triangle_{\Sigma} (F_K \circ d_{\gamma(\epsilon)}) =
  (F_K^{\prime \prime} \circ d_{\gamma(\epsilon)})
  |\nabla_{\Sigma} d_{\gamma(\epsilon)}|^2 +
  (F_K^{\prime} \circ d_{\gamma(\epsilon)}) \triangle_{\Sigma} d_{\gamma(\epsilon)},
  \end{equation}
it follows that
\begin{equation} \label{3.10}
  (F_K^{\prime} \circ d_{\gamma(\epsilon)}) \triangle_{\Sigma} d_{\gamma(\epsilon)}
  \le 4 - 
  (F_K^{\prime \prime} \circ d_{\gamma(\epsilon)})
  |\nabla_{\Sigma} d_{\gamma(\epsilon)}|^2  \le 4,
  \end{equation}
so
\begin{equation} \label{3.11}
  \triangle_{\Sigma} d_{\gamma(\epsilon)}
  \le \frac{4}{F_K^{\prime} \circ d_{\gamma(\epsilon)}},
\end{equation}
where the denominator is strictly positive in a neighborhood of $m$.

Similarly,
\begin{equation} \label{3.12}
\triangle_{\Sigma} (F_K \circ (d_{\gamma(\epsilon)} + \epsilon)) =
  (F_K^{\prime \prime} \circ (d_{\gamma(\epsilon)} + \epsilon))
  |\nabla_{\Sigma} d_{\gamma(\epsilon)}|^2 +
  (F_K^{\prime} \circ (d_{\gamma(\epsilon)} + \epsilon))
  \triangle_{\Sigma} d_{\gamma(\epsilon)}
\end{equation}
Combining with (\ref{3.10}) and (\ref{3.11}) gives
\begin{align} \label{3.13}
\triangle_{\Sigma} (F_K \circ (d_{\gamma(\epsilon)} + \epsilon)) \le &
\left( (F_K^{\prime \prime} \circ (d_{\gamma(\epsilon)} + \epsilon)) -
(F_K^{\prime \prime} \circ d_{\gamma(\epsilon)}) \right)
  |\nabla_{\Sigma} d_{\gamma(\epsilon)}|^2 + \\
&  \left(  (F_K^{\prime} \circ (d_{\gamma(\epsilon)} + \epsilon)) -
(F_K^{\prime} \circ d_{\gamma(\epsilon)})  \right)
  \triangle_{\Sigma} d_{\gamma(\epsilon)} + 4 \notag \\
  = &
\left( (F_K^{\prime \prime} \circ (d_{\gamma(\epsilon)} + \epsilon)) -
(F_K^{\prime \prime} \circ d_{\gamma(\epsilon)}) \right)
  + \notag \\
&  \left(  (F_K^{\prime} \circ (d_{\gamma(\epsilon)} + \epsilon)) -
(F_K^{\prime} \circ d_{\gamma(\epsilon)})  \right)
  \triangle_{\Sigma} d_{\gamma(\epsilon)} + 4 \notag \\
  \le &
\left( (F_K^{\prime \prime} \circ (d_{\gamma(\epsilon)} + \epsilon)) -
(F_K^{\prime \prime} \circ d_{\gamma(\epsilon)}) \right)
  + \notag \\
&  \left(  (F_K^{\prime} \circ (d_{\gamma(\epsilon)} + \epsilon)) -
(F_K^{\prime} \circ d_{\gamma(\epsilon)})  \right)
  \frac{4}{F_K^{\prime} \circ d_{\gamma(\epsilon)}}
  + 4. \notag
\end{align}
Given $\epsilon^\prime > 0$, using the continuity of
$F_K^\prime$ and  $F_K^{\prime \prime}$, by choosing $\epsilon$ small enough
we can ensure that
$\triangle_{\Sigma} (F_K \circ (d_{\gamma(\epsilon)} + \epsilon))
\le 4 + \epsilon^\prime$
in a neighborhood of $m$ in $\Sigma$. Thus
$\triangle_{\Sigma} d_{K,p}^2 \le 4$ in
the barrier sense, hence in the viscosity sense and in the
distributional sense.
This means that $\phi - d_{K,p}^2/2$ is subharmonic on $\Sigma$.
Thus (\ref{3.7}) holds.

Now suppose that $K > 0$ and
$\diam(M) = \frac{\pi}{\sqrt{2K}}$. Given $\lambda \in (0,1)$, the
metric $g$ also has $BK \ge \lambda^2 K$, while
$\diam(M) < \frac{\pi}{\sqrt{2\lambda^2K}}$.
Hence
$\phi + \frac{2}{\lambda^2 K}
\log \cos \left( \lambda d_p \sqrt{\frac{K}{2}} \right)$
is plurisubharmonic, i.e. 
$\lambda^2 \phi + \frac{2}{K}
\log \cos \left( \lambda d_p \sqrt{\frac{K}{2}} \right)$
is plurisubharmonic.
Using the fact that
$\frac{2}{K} \log \cos \left( \lambda d_p \sqrt{\frac{K}{2}} \right)$
is monotonically nonincreasing in $\lambda$ as $\lambda \rightarrow 1$,
we can pass to the limit to conclude that
$\phi + \frac{2}{K} \log \cos \left(d_p \sqrt{\frac{K}{2}} \right)$
is plurisubharmonic; c.f.
\cite[Proofs of Theorems I.4.15 and I.5.4]{Demailly (2012)}.
This proves the proposition.
\end{proof}

\begin{remark} \label{3.14}
  If $K = 0$ then Proposition \ref{3.6} was proven in
  \cite{Cao-Ni (2005)} by very different means.
  \end{remark}

\subsection{Integral comparison inequality}

We now wish to give an analog of (\ref{2.3}). Comparing (\ref{3.7}) with
(\ref{2.1}), it is clear that instead of integrating over
geodesics, i.e. real curves, we should now integrate over two
dimensional objects, i.e. complex curves.

\begin{proposition} \label{3.15}
  Let $M$ be a complete
  K\"ahler manifold. Given $K \in \R$, the manifold $M$ has $BK \ge K$
  if and only if it satisfies the following property.
  Let $i \: : \: \overline{D^2} \rightarrow M$ be an embedding of a disk
  into $M$, that is holomorphic on $D^2$. Let $\Sigma$ be the
  image of $i$. 
  Let $dA$ denote the area form on $\Sigma$.
  Let $z$ be the local coordinate on $D^2$ and let
  $\theta \in [0, 2 \pi)$ be the local coordinate on
    $\partial{\overline{D^2}}$.
  Then
  \begin{equation} \label{3.16}
    d^2_{K,p}(0) \ge \frac{2}{\pi} \iint_\Sigma \log |z| \: dA +
    \frac{1}{2\pi} \int_{\partial \Sigma} d^2_{K,p}(\theta) d\theta,
  \end{equation}
  where the ``$0$'' on the left-hand side
  denotes $i(0)$, the center of $\Sigma$.
\end{proposition}
\begin{proof}
  Suppose that $BK \ge K$. From Proposition \ref{3.6}, or more precisely its
  proof, we know that $\sqrt{-1} \partial \overline{\partial}
  d_{K,p}^2/2 \le \omega_{\Sigma}$ as currents on $\Sigma$.
    The solution to 
    $\sqrt{-1} \partial \overline{\partial} f/2 = \omega_\Sigma$
    on $\Sigma$, with
    $f \Big|_{\partial \Sigma} = d_{K,p}^2 \Big|_{\partial \Sigma}$ has
     \begin{equation} \label{3.17}
f(0) = \frac{2}{\pi} \iint_{\Sigma} \log |z| \: dA +
    \frac{1}{2\pi} \int_{\partial \Sigma} d_{K,p}^2(\theta) d\theta.
    \end{equation}
     As $f - d_{K,p}^2$ is subharmonic on $\Sigma$,
     and vanishes on $\partial \Sigma$,
     inequality (\ref{3.16}) follows.

     Now suppose that the inequality $BK \ge K$ is violated 
     at some point $p$.
     In complex normal coordinates around $p$, the metric is
     \begin{equation} \label{3.18}
       g_{i \overline{j}} = \delta_{i \overline{j}} + \frac12
       R_{i \overline{j} k \overline{l}} z^k \overline{z}^l +
       o(|z|^2),
     \end{equation}
     where $R_{i \overline{j} k \overline{l}}$ is evaluated at $p$.
     Correspondingly,
          \begin{equation} \label{3.19}
            \omega = \frac12 \sqrt{-1} dz^i \wedge d\overline{z}^i +
            \frac14 \sqrt{-1}
            R_{i \overline{j} k \overline{l}} z^k \overline{z}^l dz^i \wedge
            d\overline{z}^j +
       o(|z|^2).
       \end{equation}
     
          In general, $d^2(p_0,p_1)$ is the minimum over $\gamma$ of the energy
\begin{equation}
  E(\gamma) = \int_0^1 g_{ij} \frac{d\gamma^i}{dt} \frac{d\gamma^i}{dt} \: dt,
  \end{equation}
where $\gamma : [0,1] \rightarrow M$ has $\gamma(0) = p_0$ and
$\gamma(1) = p_1$. If $\gamma$ is a unique minimizer and we perturb the metric
by $\delta g$ then to leading order, the squared distance changes by
\begin{equation}\label{3.19.5}
  \delta d^2(p_0, p_1) = \int_0^1 \delta g_{ij} \frac{d\gamma^i}{dt} \frac{d\gamma^i}{dt} \: dt.
  \end{equation}

In our case, for the flat metric the minimizer between $0 \in \C^n$ and
$z \in \C^n$ is $\gamma(t) = tz$. Treating the second term in
(\ref{3.18}) as the perturbation, the change in squared distance is
\begin{equation} \label{3.20}
\frac12 \int_0^1
  R_{i \overline{j} k \overline{l}}
   z^i z^{\overline{j}} (tz^k) (t\overline{z}^l)\: dt  =
\frac16 \: 
  R_{i \overline{j} k \overline{l}} z^i z^{\overline{j}}
  z^k \overline{z}^l  
\end{equation}
Hence since $p = 0$ in the local coordinates,
\begin{equation} \label{3.21}
  d_p^2(z) = |z|^2  + \frac16
  R_{i \overline{j} k \overline{l}}
  z^i z^{\overline{j}} z^k \overline{z}^l + o(|z|^4).
\end{equation}
From (\ref{1.1}),
\begin{equation} \label{3.22}
  d^2_{K,p} = d_p^2 + \frac{1}{12} K d_p^4 + o(d_p^4),
\end{equation}
so
\begin{equation} \label{3.23}
    d_{K,p}^2(z) = |z|^2  + \frac16
  R_{i \overline{j} k \overline{l}}
  z^i z^{\overline{j}} z^k \overline{z}^l
+ \frac{1}{12} K |z|^4
  + o(|z|^4).
  \end{equation}
This gives
\begin{align} \label{3.24}
  \sqrt{-1} \partial \overline{\partial} d_{K,p}^2/2 = &
  \frac12 \sqrt{-1} dz^i \wedge d\overline{z}^i +
  \frac13 \sqrt{-1}   R_{i \overline{j} k \overline{l}}
  z^k z^{\overline{l}} dz^i \wedge d\overline{z}^j +
  \\
  & \frac{1}{12} \sqrt{-1} K \overline{z}^i z^j dz^i \wedge d\overline{z}^j +
  \frac{1}{12} \sqrt{-1} K |z|^2 dz^i \wedge d\overline{z}^i +   
  o(|z|^2). \notag
\end{align}
Equations (\ref{3.19}) and (\ref{3.24}) give
\begin{equation} \label{3.25}
  \sqrt{-1} \partial \overline{\partial} d_{K,p}^2/2 - \omega  =
  \frac{1}{12} \sqrt{-1}
R^\prime_{i \overline{j} k \overline{l}} 
  z^k z^{\overline{l}} dz^i \wedge d\overline{z}^j + o(|z|^2),
\end{equation}
where
\begin{equation} \label{3.26}
  R^\prime_{i \overline{j} k \overline{l}} = 
    R_{i \overline{j} k \overline{l}} + K
    (\delta_{i\overline{j}} \delta_{k\overline{l}} +
    \delta_{i\overline{l}} \delta_{\overline{j}k}).
\end{equation}

If $\Sigma$ is an embedded holomorphic disk in $M$ then
\begin{equation} \label{3.27}
    d_{K,p}^2(0) - \frac{2}{\pi} \iint_\Sigma \log |z| \: dA -
    \frac{1}{2\pi} \int_{\partial \Sigma} d_{K,p}^2(\theta) d\theta =
    \frac{2}{\pi} \iint_\Sigma \log |z| \: \left(
      \sqrt{-1} \partial \overline{\partial} d_{K,p}^2/2 - \omega \right).
\end{equation}

Since $M$ does not have $BK \ge K$
at $p$, there are unit vectors $X,Y \in T_p^{(1,0)}M$ so that
$R^\prime(X,\overline{X},Y, \overline{Y}) > 0$.
(Recall the minus sign in (\ref{3.3}).)

Given $0 < \epsilon_1 < < 
\epsilon_2 < < 1$,
consider a holomorphic
disk $i \: : \: \overline{D^2} \rightarrow M$ given in complex normal
coordinates by
$i(w) = \epsilon_1 w X + \epsilon_2 Y$. 
Let $\Sigma$ be the image of $i$.
Using (\ref{3.25}), the right-hand side of (\ref{3.27}) is
approximately
\begin{align} \label{3.28}
&  \frac{1}{6 \pi} \: \sqrt{-1} \epsilon_1^2 \epsilon_2^2
 (\log \epsilon_2) \:  R^\prime(X,\overline{X},Y, \overline{Y})
  \iint_{D^2} dw \wedge d\overline{w} = \\
&  \frac{1}{3\pi} \: \epsilon_1^2 \epsilon_2^2  (\log \epsilon_2) \:
  R^\prime(X,\overline{X},Y, \overline{Y})
  \iint_{D^2} dA_{D^2}.\notag
\end{align}
Since 
$\log \epsilon_2 < 0$, we conclude that
\begin{equation} \label{3.29}
    d_{K,p}^2(0) - \frac{2}{\pi} \iint_\Sigma \log |z| \: dA -
    \frac{1}{2\pi} \int_{\partial \Sigma} d_{K,p}^2(\theta) d\theta < 0,
\end{equation}
contradicting (\ref{3.16}).
\end{proof}

\begin{remark} \label{3.30}
  There is an analogy between (\ref{2.3}), with $t = \frac{L}{2}$, and
  (\ref{3.16}), where $\frac12 \left( d_p^2(L) + d_p^2(0) \right)$
  is replaced by
  $\frac{1}{2\pi} \int_{\partial \Sigma} d^2_{K,p}(\theta) d\theta$ and
  $- \frac{L^2}{4}$ is replaced by
$\frac{2}{\pi} \iint_\Sigma \log |z| \: dA$.
    
  For any point $q$ in the disk, there is an inequality similar to
  (\ref{3.16}) with $0$ replaced by $q$, obtained by performing a
  holomorphic automorphism of the disk.

Note that the area form $dA$ in (\ref{3.16}) can also be described as the
two-dimensional Hausdorff measure on $\Sigma$. Hence the statement of
(\ref{3.16}) only depends on the complex structure and the metric $d$.
\end{remark}

\subsection{Hermitian manifolds}

One can ask when (\ref{3.16}) holds more generally in the setting of
Hermitian manifolds, rather than K\"ahler manifolds.
It turns out that if (\ref{3.16}) holds for a Hermitian manifold then it
is forced to be K\"ahler.
We now give an analog of Remark \ref{2.4}(4), in which Finsler manifolds are
replaced by Hermitian manifolds, and Riemannian manifolds are replaced by
K\"ahler manifolds.

\begin{proposition} \label{3.31}
  If a Hermitian manifold $M$ satisfies (\ref{3.16}), for all $p \in M$ and
  all holomorphic disks $\Sigma$, then it is K\"ahler.
\end{proposition}
\begin{proof}
Choose complex coordinates around $p$. 
After a change of coordinates, we can write the metric locally as
\begin{equation} \label{3.32}
g = dz^i d\overline{z}^i + T_{\overline{i}jk} z^j dz^k d\overline{z}^i +
\overline{T_{\overline{i}jk}} \overline{z}^j d\overline{z}^k dz^i + O(|z|^2).
\end{equation}
Here $T_{\overline{i}jk}$ is a constant times the torsion tensor at $p$, and is
antisymmetric in $j$ and $k$.

We first compute the leading order terms in $d_p^2$, using (\ref{3.19.5}).
For the flat metric the minimizer between $0 \in \C^n$ and
$z \in \C^n$ is $\gamma(t) = tz$. Treating the second and third terms in
(\ref{3.32}) as the perturbation, the change in squared distance is
\begin{equation}
\int_0^1 (T_{\overline{i}jk}) (tz^j) z^k \overline{z}^i \: dt +
complex \: conjugate.
\end{equation}
  This would be the $O(|z|^3)$ term in $d_p^2$, but it vanishes
  because of the $(jk)$-antisymmetry of $T_{\overline{i}jk}$.
  Hence $d_p^2(z) = |z|^2 + O(|z|^4)$. From (\ref{3.22}), it follows that
  $d_{K,p}^2(z) = |z|^2 + O(|z|^4)$.
  
  Then
\begin{equation} \label{3.33}
\sqrt{-1} \partial \overline{\partial} d_{K,p}^2/2 =
  \frac12 \sqrt{-1} dz^i \wedge d\overline{z}^i +
  O(|z|^2).
\end{equation}
On the other hand,
\begin{equation} \label{3.34}
  \omega = \frac12 \sqrt{-1} dz^i \wedge d\overline{z}^i +
  \frac12 \sqrt{-1} T_{\overline{i}jk} z^j dz^k \wedge d\overline{z}^i +
  \frac12 \sqrt{-1} \:
  \overline{T_{\overline{i}jk}} \overline{z}^j d\overline{z}^k
  \wedge dz^i + O(|z|^2),
\end{equation}
so
\begin{equation} \label{3.35}
  \sqrt{-1} \partial \overline{\partial} d_{K,p}^2/2 - \omega =
  - \:   \frac12 \sqrt{-1} T_{\overline{i}jk} z^j dz^k \wedge d\overline{z}^i
  - \: 
  \frac12 \sqrt{-1} \:
  \overline{T_{\overline{i}jk}} \overline{z}^j d\overline{z}^k
  \wedge dz^i + O(|z|^2).
\end{equation}

Suppose that $M$ is nonK\"ahler, so it has a nonzero torsion tensor at
some point $p$.
Let $\vec{b} \in \C^n$ be such that $\sum_j b^j T_{\overline{i}jk}$ is
a nonzero matrix in $(\overline{i}, k)$. Let 
$\vec{a} \in \C^n$ be such that
$\sum_{i,j,k} \overline{a^i} b^j T_{\overline{i}jk} a^k \neq 0$.
Multiplying $\vec{b}$ by a constant, we can assume that
$\sum_{i,j,k} \overline{a^i} b^j T_{\overline{i}jk} a^k$ is
a negative real number.
Given $0 < \epsilon_1 < < 
\epsilon_2 < < 1$,
consider a small disk $i \: : \: \overline{D^2} \rightarrow M$ given by
$i(w) = \epsilon_1 w \vec{a} + \epsilon_2 \vec{b}$. 
Let $\Sigma$ be the image of $i$. As in the proof of Proposition \ref{3.15},
it follows from (\ref{3.35}) that the right-hand side of (\ref{3.27})
is approximately
\begin{equation}
  -4 \epsilon_1^2 \epsilon_2 \log(\epsilon_2 |\vec{b}|)
  \sum_{i,j,k} \overline{a^i} b^j T_{\overline{i}jk} a^k < 0.
  \end{equation}
Thus (\ref{3.16}) is violated for $\Sigma$, which is a contradiction.
\end{proof}

\subsection{Domains in model spaces}

We now give an analog of Remark \ref{2.4}(5). That is, we look at regions in
$\C^n$ or, more generally, in model spaces of constant holomorphic
sectional curvature.
Since we want to characterize when (\ref{3.16}) holds, we need a complex
structure everywhere.  For that reason, we do not allow boundary, but
simply consider when a domain in the model space satisfies (\ref{3.16}).
One might initially expect that it has something with pseudoconvexity
of the domain.
However, the latter notion is invariant under biholomorphisms, whereas we
have a metric $d$ in addition.  It turns out that the answer is essentially
given by convexity in the usual sense.

Given $K \in \R$, let $M_K$ be the complete simply connected K\"ahler manifold
with constant holomorphic sectional
curvature $2K$. Its metric is given by
(\ref{3.4}), with $c = 2K$. One can check that equality is achieved in
(\ref{3.7}), away from the cut locus of $p$ if $K > 0$.

\begin{proposition} \label{3.36}
  Let $M$ be a connected open subset of $M_K$. Let $d$ be the length
  metric on $M$. 
  Then $M$ satisfies
  (\ref{3.16}) if and only if $d$ coincides with the restriction
  ${\mathcal D}$ of the
  metric from $M_K$.
\end{proposition}
\begin{proof}
  If $d = {\mathcal D}$ then (\ref{3.16}) follows immediately from the
  corresponding inequality for $M_K$.
  
  Suppose that (\ref{3.16}) is satisfied for $M$, but
  $d \neq {\mathcal D}$.
  Let $m_1, m_2 \in M$ be points such that
  $d(m_1,m_2) > {\mathcal D}(m_1,m_2)$.
  If $K > 0$, let $D$ denote the cut locus of $m_1$, a copy of $\C P^{n-1}$.
  By continuity of the distance functions, we can assume that
  $m_2 \notin D$.

  Let $\gamma : [0,1] \rightarrow M$ be a smooth embedding with
  $\gamma(0) = m_1$ and $\gamma(1) = m_2$.
  If $K>0$ then we can assume that $\gamma$ is disjoint from $D$.
    By approximation, we can
  assume that $\gamma$ is real analytic. We can then extend $\gamma$ to a
  real analytic embedding $\gamma : [- \epsilon,1+\epsilon] \rightarrow M$
    for some $\epsilon > 0$.

    We claim that
    after possibly reducing $\epsilon$, there is some $\epsilon^\prime > 0$, and a
    continuous embedding
    $\Gamma : [- \epsilon,1+\epsilon] \times
    [-\epsilon^\prime, \epsilon^\prime] \rightarrow M$ that is
    holomorphic on the interior, so that
    $\Gamma(t,0) = \gamma(t)$ for all
    $t \in [- \epsilon,1+\epsilon]$. To see this, suppose
    first that $K = 0$, so $M_K = \C^n$. Let $\{\gamma^i(t)\}_{i=1}^n$
    be the components of $\gamma$. As $\gamma^i$ is real analytic, it
    extends to a holomorphic function $\Gamma^i : 
    (- \epsilon,1+\epsilon) \times
    (-\epsilon_i^\prime, \epsilon_i^\prime) \rightarrow \C$ for some
    $\epsilon^\prime_i > 0$. Taking
    $\epsilon^\prime = \min_i \epsilon^\prime_i$,
    the functions $\{\Gamma^i\}_{i=1}^n$
    combine to give a holomorphic map $\Gamma :
    (- \epsilon,1+\epsilon) \times
    (-\epsilon^\prime, \epsilon^\prime) \rightarrow \C^n$. The
    image of $d\Gamma_{(t,0)}$ is the span of $\gamma^\prime(t)$ and
    $J \gamma^\prime(t)$, a two dimensional space.
    Hence by reducing $\epsilon$ and $\epsilon^\prime$, we can ensure that
    $\Gamma$ is a continuous embedding from 
    $[- \epsilon,1+\epsilon] \times
    [-\epsilon^\prime, \epsilon^\prime]$ to $M$,
    which is holomorphic on the interior.

    If $K < 0$ then the underlying complex structure of $M_K$ is the unit
    ball in $\C^n$, so the same argument can be applied.  If $K > 0$ then
    $M_K - D$ is biholomorphic to $\C^n$, so again the same argument can be
    applied.

    As $\Gamma$ reparametrizes to a holomorphic disk
    $i : \overline{D^2} \rightarrow M$ with image $\Sigma$,
    by a holomorphic automorphism
    of the disk we can assume that $i(0) = m_1$.
    The equality case of
    (\ref{3.7}) with $p = m_1$ implies
    \begin{equation} \label{3.37}
    0 = \frac{2}{\pi} \iint_\Sigma \log |z| \: dA +
    \frac{1}{2\pi} \int_{\partial \Sigma} {\mathcal D}^2_{K,m_1}(\theta) d\theta.
    \end{equation}
    Note that the two dimensional Hausdorff measure $dA$ is the same for
    $d$ and ${\mathcal D}$.
    Since $d(m_1, m_2) > {\mathcal D}(m_1, m_2)$, if
    $\epsilon$ and $\epsilon^\prime$ are small enough then
    $d^2_{K,m_1}(\theta) > {\mathcal D}^2_{K,m_1}(\theta)$ for some
    $\theta$. By continuity of the distance functions, this will also be
    true for all $\theta$ in some open interval.  Thus
        \begin{equation} \label{3.38}
    0 < \frac{2}{\pi} \iint_\Sigma \log |z| \: dA +
    \frac{1}{2\pi} \int_{\partial \Sigma} d^2_{K,m_1}(\theta) d\theta,
        \end{equation}
        which contradicts (\ref{3.16}).
\end{proof}

\section{Noncollapsed Gromov-Hausdorff limits} \label{sect4}

We consider a noncollapsed pointed Gromov-Hausdorff
limit of a sequence of complete K\"ahler
manifolds with $BK \ge K$. Lee and Tam proved 
that the limit has
the structure of a complex manifold \cite{Lee-Tam (2019)}.
This extends earlier results of
Liu \cite{Liu (2018),Liu (2019)}, and is an analog of
Remark \ref{2.4}(2). We wish to study the geometry of the limit.
Although the metric $d$ on the
limit is generally not smooth, we show that it satisfies the
comparison inequality (\ref{3.16}).
This is an analog of
Remark \ref{2.4}(1).

The method of proof is by running the Ricci flow on the
approximants and passing to a limiting Ricci flow
that exists for positive time (locally). Then one is reduced to
understanding the $t \rightarrow 0$ limit of a single Ricci flow,
as opposed to a sequence of Riemannian manifolds.  This approach has
been applied in many other contexts. Since we are not assuming an
upper curvature bound, we apply recent results on local Ricci flow.

The proof also relies on local K\"ahler potentials. We actually prove the
existence of local K\"ahler potentials, of a certain regularity, on
the limit space.

\begin{proposition} \label{4.1}
Let $\{(M_i, p_i, g_i)\}_{i=1}^\infty$ be a sequence of pointed
$n$-dimensional complete K\"ahler manifolds with
$BK \ge K$.
Suppose that there is some
$v_0 > 0$ so that for all $i$, we have $\vol(B(p_i, 1)) \ge v_0$. Then after
passing to a subsequence, there is a pointed Gromov-Hausdorff limit
$(X_\infty, p_\infty, d_\infty)$ with the following properties.
\begin{enumerate}
\item $X_\infty$ is a complex manifold and $d_\infty$ is locally
  biH\"older-equivalent to the distance metric of a smooth Riemannian metric on
  $X_\infty$.
\item There is an open covering $\{U_\alpha\}_{\alpha \in A}$ of $X_\infty$
and plurisubharmonic potentials
$\phi_\alpha \in C(U_\alpha)$, locally Lipschitz with respect to $d_\infty$,
so that $\phi_\alpha - \phi_\beta$ is pluriharmonic on
$U_\alpha \cap U_\beta$, and the following holds.
Let $\Sigma$ be a holomorphic disk in $X_\infty$.
Let $\phi_\alpha
\Big|_{\Sigma \cap U_\alpha}$ be the restriction of $\phi_\alpha$ to
$\Sigma \cap U_\alpha$ and put
 $\omega_\infty \Big|_{\Sigma} =
\sqrt{-1} \partial \overline{\partial} \phi_\alpha
\Big|_{\Sigma \cap U_\alpha}$, a
globally defined measurable $(1,1)$-form on $\Sigma$.
Then 
$\omega_\infty \Big|_{\Sigma}$ equals the two dimensional Hausdorff measure
$\mu_\infty$
coming from $d_\infty \Big|_{\Sigma}$.
\item We have
  \begin{equation} \label{4.2}
d_{K,p}^2(0) \ge 
\frac{2}{\pi} \int_{\Sigma} \log |z| \: d\mu_\infty +
\frac{1}{2\pi} \int_{\partial \Sigma}
d_{K,p}^2 \left( \theta \right)
d\theta.
  \end{equation}
 \end{enumerate}
\end{proposition}
\begin{proof}
(1). We claim first that there are nondecreasing sequences
  $\alpha_k, \beta_k \ge 1$ and a nonincreasing sequence $S_k > 0$ such that
  for any $i$, there is a K\"ahler-Ricci flow
  $g_{i}(t)$ defined on 
$\bigcup_{k=1}^\infty (B_{g_i}(p_i, 2k) \times [0, S_k])$
  with $g_{i}(0) = g_i$, such that
  \begin{equation} \label{4.3}
\left| \Rm \left( g_{i}(t) \right) \right| \le \frac{\alpha_k}{t},
  \end{equation}
  \begin{equation} \label{4.4}
    \Ric(g_i(t)) \ge - \beta_k
    \end{equation}
  and
    \begin{equation} \label{4.5}
\inj_{g_{i}(t)}  \ge \alpha_k^{-1} \sqrt{t}
    \end{equation}
    on $B_{g_i}(p_i, 2k) \times [0, S_k]$.
  This follows from the pyramid Ricci flow
  constructed in \cite[Theorem 1.2]{Lee-Tam (2019)}
  (see also 
  the proof of
  \cite[Theorem 5.1]{Lee-Tam (2017)} and the proof of
  \cite[Theorem 1.3]{McLeod-Topping (2018)}).

  From distance distortion estimates as in
  \cite[Section 27]{Kleiner-Lott (2008)},
there is then a constant
$C_k < \infty$ so that for $t_1 \le t_2$, we have
\begin{equation} \label{4.6}
  d_{g_i(t_1)} - C_k (\sqrt{t_2} - \sqrt{t_1}) \le d_{g_i(t_2)} \le
  e^{\beta_k (t_2 - t_1)} d_{g_i(t_1)}
\end{equation}
on $B_{g_i}(p_i, 2k) \times [0, S_k]$.
  
Using a local version of Hamilton compactness
\cite[Appendix E]{Kleiner-Lott (2008)},
  after passing to a subsequence of the $i$'s, there is a
  pointed smooth manifold $(X_\infty, p_\infty)$ and an exhaustion of
  $X_\infty$ by precompact open sets $\{V_k\}_{k=1}^\infty$ containing
  $p_\infty$, along with
  a limiting
  pointed Ricci flow
  $g_\infty(\cdot)$ defined on
  $\bigcup_{k=1}^\infty (V_k \times (0, S_k))$;
  c.f. \cite[Theorem 1.5]{McLeod-Topping (2018)}.
  More precisely, for each $k \in \Z^+$,
    for large $i$ there is a pointed embedding
    $\phi_{i,k} : V_k \rightarrow
    M_i$
so that
\begin{equation} \label{4.7}
  g_\infty(\cdot) = \lim_{i \rightarrow \infty} \phi_{i,k}^*
  g_i(\cdot)
\end{equation}
on compact subsets of $V_k \times (0, S_k)$,
  in the smooth topology.

  The distance distortion estimate (\ref{4.6}) passes to the
  limiting Ricci flow.
It follows that there is a pointed Gromov-Hausdorff limit
$\lim_{t \rightarrow 0} (X_\infty, p_\infty, g_\infty(t)) =
(X_\infty, p_\infty, d_\infty)$ for some complete metric $d_\infty$.
It then follows that
$\lim_{i \rightarrow \infty} (M_i, p_i, g_i) =
(X_\infty, p_\infty, d_\infty)$ in the pointed Gromov-Hausdorff topology.
We can take $V_k$ to be the metric ball $B(p_\infty, k)$ with respect to
$d_\infty$, so
\begin{equation} \label{4.8}
  d_{g_\infty(t_1)} - C_k (\sqrt{t_2} - \sqrt{t_1}) \le d_{g_\infty(t_2)} \le
  e^{\beta_k (t_2 - t_1)} d_{g_\infty(t_1)}
\end{equation}
on $B(p_\infty, k) \times (0, S_k)$.
Also,
\begin{equation} \label{4.9}
\left| \Rm \left( g_{\infty}(t) \right) \right| \le \frac{\alpha_k}{t}
\end{equation}
on $B(p_\infty, k) \times (0, S_k)$.

From \cite[Lemma 3.1]{Simon-Topping (2017)}, for any $t \in (0, S_k)$,
    the metric
    ball $B(p_\infty, k) \subset X_\infty$ with the metric $d_\infty$
    is 
    biH\"older homeomorphic to the same ball with the metric $g_\infty(t)$.

  Given $k \in \Z^+$ and considering the time interval $(0, S_k)$, since
    the complex structures $J_i$ on
    $B_{g_i}(p_i, 2k) \subset M_i$ satisfy $\nabla_{g_i(t)} J_i = 0$,
    after passing to a subsequence of $i$'s
    we can assume that they converge to
    a complex structure $J_{\infty,k}$ on $B(p_\infty, k)$ that satisfies
    $\nabla_{g_\infty(t)} J_{\infty,k} = 0$. After passing to a
    further subsequence of
    $i$'s, we obtain
    a complex structure $J_\infty$ on $X_\infty$
    that, on $B(p_\infty, k)$, satisfies
    $\nabla_{g_\infty(t)} J_\infty = 0$ for $t \in (0, S_k)$.
      Let $\omega(t)$ denote the corresponding K\"ahler form.
       \\ \\
(2).  Fix $k \in \Z^+$ and fix $t^\prime \in (0, S_k)$.
      For $t \in (0, t^\prime]$, put
      \begin{equation} \label{4.10}
        u(t) = - \int_t^{t^\prime} \log \frac{\omega^n(s)}
        {\omega^n(t^\prime)} \: ds.
        \end{equation}
Then
\begin{equation} \label{4.11}
\omega(t) = \omega(t^\prime) - (t-t^\prime) \Ric(\omega(t^\prime))
+ \sqrt{-1} \partial \overline{\partial}
u(t),
\end{equation}
as can be seen by differentiating in $t$.

Since
\begin{equation} \label{4.12}
  \frac{\partial \omega}{\partial t} \: = \: - \: \Ric(\omega(t)),
\end{equation}
the estimate (\ref{4.9}) implies
\begin{equation} \label{4.13}
\left| \log \frac{\omega^n(s)}{\omega^n(t^\prime)} \right| \le \const \log
\frac{t^\prime}{s}
\end{equation}
for $s \in (0, t^\prime]$, where ``$\const$'' is an $n$-dependent factor times
  $\alpha_k$.
Then
\begin{align} \label{4.14}
  |u(t_1) - u(t_2)| & \le \const \int_{t_1}^{t_2} \log
  \frac{t^\prime}{s} \: ds \\
  & = \const \left( (t_2 - t_1) \log(t^\prime) -
  t_2 \log(t_2) + t_1 \log(t_1) \right). \notag
\end{align}
Hence $\{ u(1/j)\}$ is a uniformly Cauchy sequence and has a limit
$u(0) \in C(B(p_\infty, k))$.

Given $x \in B(p_\infty, k)$, let $U$ be a neighborhood of $x$ that is
biholomorphic to the unit ball in $\C^n$.
There are $v_U, w_U \in C^\infty(U)$ so that
we can write
$\omega(t^\prime)$ on $U$ as $\sqrt{-1} \partial \overline{\partial} v_U$,
and we can write $\Ric(\omega(t^\prime))$ on $U$ as
$\sqrt{-1} \partial \overline{\partial} w_U$.
Doing the same for another point $p^\prime \in B(p_\infty, k)$, we have
$\sqrt{-1} \partial \overline{\partial} (v_U - v_{U^\prime}) = 0$ and
$\sqrt{-1} \partial \overline{\partial} (w_U - w_{U^\prime}) = 0$ on
$U \cap U^\prime$.
For $t \in [0, S_k)$, put
\begin{equation} \label{4.15}
  \phi_U(t) = v_U - (t-t^\prime) w_U + u(t) \Big|_U.
\end{equation}
If $t > 0$ 
then (\ref{4.11}) gives
$\sqrt{-1} \partial \overline{\partial} \phi_U(t) = \omega(t)$, so
$\sqrt{-1} \partial \overline{\partial} (\phi_U(t) - \phi_{U^\prime}(t)) = 0$
on $U \cap U^\prime$. Let $\eta \in \Omega^{n-1,n-1}(U \cap U^\prime)$ be
a smooth compactly supported form. Then
\begin{equation} \label{4.16}
  \int_{X_\infty} (\phi_U(t) - \phi_{U^\prime}(t)) \wedge
  \sqrt{-1} \partial \overline{\partial} \eta =
  \int_{X_\infty} \sqrt{-1} \partial \overline{\partial}
  (\phi_U(t) - \phi_{U^\prime}(t))
  \wedge \eta = 0.
\end{equation}
Using the uniform convergence $\lim_{t \rightarrow 0} u(t) = u(0)$,
it follows that
\begin{equation} \label{4.17}
  \int_{X_\infty} (\phi_U(0) - \phi_{U^\prime}(0)) \wedge
  \sqrt{-1} \partial \overline{\partial} \eta = 0,
\end{equation}
so
$\sqrt{-1} \partial \overline{\partial} (\phi_U(0) - \phi_{U^\prime}(0)) = 0$
as a current.  That is, $\phi_U(0) - \phi_{U^\prime}(0)$ is pluriharmonic.
Similarly, if $\eta$ has compact support in $U$ and is strongly positive
in the sense of
\cite[Chapter 3]{Demailly (2012)} then for $t > 0$, we have
\begin{equation} \label{4.18}
  \int_{X_\infty} \phi_U(t) \wedge \sqrt{-1} \partial \overline{\partial} \eta =
  \int_{X_\infty} \sqrt{-1} \partial \overline{\partial} \phi_U(t) \wedge \eta =
  \int_{X_\infty} \omega(t) \wedge \eta \ge 0.
\end{equation}
Passing to the limit as $t \rightarrow 0$ gives
\begin{equation} \label{4.19}
  \int_{X_\infty} \phi_U(0) \wedge \sqrt{-1} \partial \overline{\partial} \eta
  \ge 0.
\end{equation}
Hence $\sqrt{-1} \partial \overline{\partial} \phi_U(0) \ge 0$ in the sense of
currents,
i.e. $\phi_U(0)$ is plurisubharmonic.

From \cite[Theorem 6]{Cheng-Yau (1975)},
there is a bound on $|\nabla \phi_U(t)|$
in terms of $K$ and
the oscillation of $\phi_U(t)$, the latter of which is uniformly bounded in
$t$. Hence $\phi_U(t)$ is uniformly Lipschitz in $t$, with respect to
$d_{g_\infty(t)}$. This passes to the
limit, to show that $\phi_U(0)$ is Lipschitz with respect to $d_\infty$.

Taking an open cover $\{U_\alpha\}$
of $X_\infty$ by such neighborhoods, we obtain such plurisubharmonic functions
$\phi_\alpha = \phi_{U_\alpha}(0) \in C(U_\alpha)$ so that
$\phi_\alpha - \phi_\beta$ is pluriharmonic on $U_\alpha \cap U_\beta$.

Fixing $k$, for $t \in (0, S_k)$
put $\widehat{d}_t= e^{- \beta_k t} d_{g_\infty(t)}$.
From (\ref{4.8}), we know that
$\widehat{d}_t$ is nonincreasing in $t$.
In addition, it follows from (\ref{4.8}) that
\begin{equation} \label{4.20}
\widehat{d}_t \le d_\infty \le e^{\beta_k t} \widehat{d}_t + C_k \sqrt{t}.
\end{equation}

Let $\Sigma$ be a holomorphic disk in $B(p_\infty, k)$. Then
for $t \in (0, S_k)$, the two dimensional Hausdorff measure
  $\widehat{\mu}_t$ on $\Sigma$
  coming from $\widehat{d}_t \Big|_{\Sigma}$ is $e^{-2\beta_kt}$ times
    $\omega(t) \Big|_{\Sigma} = \sqrt{-1} \partial \overline{\partial}
    \phi_U(t) \Big|_{\Sigma}$.
It follows that $\lim_{t \rightarrow 0} \widehat{\mu}_t$ equals 
$\sqrt{-1} \partial \overline{\partial}
    \phi_U(0) \Big|_{\Sigma} = \omega_\infty \Big|_{\Sigma}$.

    We claim that $\lim_{t \rightarrow 0} \widehat{\mu}_t$ also equals
$\mu_\infty$, the two dimensional Hausdorff measure
    coming from $d_\infty \Big|_{\Sigma}$. To see this,
    let $K \subset \Sigma$ be a compact set lying in some
    $B(p_\infty, k)$. Then
$\mu_\infty(K) = \lim_{\delta \rightarrow 0}
H^2_{d_\infty, \delta}(K)$, where
    \begin{equation} \label{4.21}
H^2_{d_\infty, \delta}(K) = \frac{\pi}{4} \inf \sum_l (\diam_{d_\infty} W_l)^2,
    \end{equation}
and $\{W_l\}$ ranges over finite covers of $K$ by open sets $W_l \subset
    \Sigma$ with
    $\diam_{d_\infty}(W_l) < \delta$.
    The definition of $\widehat{\mu}_t$ is similar, using
    $\widehat{d}_t$.
    Note that $H^2_{d_\infty, \delta}(K)$ is
    nonincreasing in $\delta$.
    Since
    $\widehat{d}_t$ is monotonically nondecreasing as $t \rightarrow 0$, with
    limit $d_\infty$, it follows from (\ref{4.21}) that
    $\widehat{\mu}_t(K)$ is monotonically nondecreasing
    as $t \rightarrow 0$, and
    $\lim_{t \rightarrow 0} \widehat{\mu}_t(K) \le \mu_\infty(K)$. To show
    equality, suppose first that $\mu_\infty(K) < \infty$.
    Given $t$, $\delta$ and $\epsilon$, let
    $\{W_l\}$ be a finite open cover of $K$ with
\begin{equation} \label{4.22}
  \frac{\pi}{4} \sum_l \left( \diam_{\widehat{d}_t} W_l \right)^2 \le
  H^2_{\widehat{d}_t, \delta}(K)
  + \epsilon,
\end{equation}
and $\diam_{\widehat{d}_t} W_l < \delta$ for each $l$.
Now
\begin{equation} \label{4.23}
\frac{\pi}{4} \sum_l \left( \diam_{d_\infty} W_l \right)^2 \le
\frac{\pi}{4}
\sum_l \left( e^{\beta_k t} \diam_{\widehat{d}_t} W_l + C_k \sqrt{t} \right)^2
\end{equation}
and $\diam_{d_\infty} W_l <
e^{\beta_k t} \delta + C_k \sqrt{t}$ for each $l$.
Since $\{W_l\}$ is finite, if $t$ is small enough then
\begin{equation} \label{4.24}
  \frac{\pi}{4} \sum_l \left( e^{\beta_k t}
  \diam_{\widehat{d}_t} W_l + C_k \sqrt{t} \right)^2 \le
  \frac{\pi}{4} \sum_l \left( \diam_{\widehat{d}_t} W_l \right)^2 + \epsilon.
\end{equation}
Put $\delta^\prime = e^{\beta_k t} \delta + C_k \sqrt{t}$. Then
\begin{equation} \label{4.25}
  H^2_{d_\infty, \delta^\prime}(K) \le
  H^2_{\widehat{d}_t, \delta}(K) + 2 \epsilon \le
  \widehat{\mu}_t(K) + 2 \epsilon \le
  \lim_{t^\prime \rightarrow 0} \widehat{\mu}_{t^\prime}(K) + 2 \epsilon.
\end{equation}
As $\epsilon$ is arbitrary, this shows that
$H^2_{d_\infty, \delta^\prime}(K) \le 
\lim_{t^\prime \rightarrow 0} \widehat{\mu}_{t^\prime}(K)$.
A similar argument shows that if $\mu_\infty(K) = \infty$
then $\lim_{t^\prime \rightarrow 0} \widehat{\mu}_{t^\prime}(K) = \infty$.
Hence
$\mu_\infty \le \lim_{t^\prime \rightarrow 0} \widehat{\mu}_{t^\prime}$. \\ \\
(3).
Given $p \in X_\infty$, let $d_p \in C(X_\infty)$
be the distance function
from $p$. Given $x \in X_\infty$, choose
$k \in \Z^+$ so that $x \in B(p_\infty, k/2)$. Let $U
\subset B(p_\infty, k/2)$ be a ball neighborhood of
$x$ on which the potential function $\phi_U(0) \in C(U)$ is defined.

  Using the comparison maps in (\ref{4.7}), we can assume that
  each Ricci flow $g_i(\cdot)$ is defined on
  $B(p_\infty, k) \times (0, S_k)$. As $\lim_{i \rightarrow \infty}
  J_i = J_\infty$ smoothly (say relative to $g_\infty(t^\prime)$
  for a given $t^\prime \in (0, S_k)$),
  there is a sequence of holomorphic maps
  $\mu_{i} : (U, J_\infty) \rightarrow
  (B(p_\infty, k), J_i)$, for large $i$, with $\{\mu_{i}\}_{i=1}^\infty$
    smoothly approaching the identity map
  \cite{Hamilton (1977)}. The 
  pullback Ricci flows $\{\mu_i^* g_i(\cdot)\}_{i=1}^\infty$ live on
  $U$ and are all K\"ahler relative to the
  fixed complex structure $J_\infty$.

  Let $\{p_i\}_{i=1}^\infty$ be a sequence of points, with $p_i \in M_i$, that
  converges to $p$ in the Gromov-Hausdorff sense. We first 
  show that $\lim_{i \rightarrow \infty} \mu_i^* d_{p_i} = d_p$
  uniformly on $U$.
  To see this, we apply (\ref{4.6}) with $t_1 = 0$ and $t_2 = t$
  to get that
  for all $q \in U$, we have
  \begin{equation} \label{newdist}
    e^{-\beta_k t} d_{g_i(t)}(q, \mu_i(q)) \le
    d_i(q, \mu_i(q)) \le
    d_{g_i(t)}(q, \mu_i(q)) + C_k \sqrt{t}.
  \end{equation}
  For fixed $t$, we have
  $\lim_{i \rightarrow \infty} d_{g_i(t)}(q, \mu_i(q)) = 0$
  uniformly in $q$. Taking $t$ to zero, we conclude
  from (\ref{newdist}) that
  $\lim_{i \rightarrow \infty} d_i(q, \mu_i(q)) = 0$ uniformly in $q$.
  Now
  \begin{equation} \label{triangle}
  |(\mu_i^*d_{p_i})(q) - d_p(q)| = |d_i(p_i, \mu_i(q)) - d_\infty(p, q)| \le
  |d_i(p_i, q)) - d_\infty(p, q)| +
  |d_i(q, \mu_i(q))|.
  \end{equation}
  Using the Gromov-Hausdorff convergence of $d_i$ to $d_\infty$,
  relative to the identity comparison map, equation (\ref{triangle}) gives
  that $\lim_{i \rightarrow \infty} \mu_i^* d_{p_i} = d_p$
  uniformly on $U$.

  We will show that there are local K\"ahler potentials
  $\{\eta_i\}$ on $M_i$ so that $\lim_{i \rightarrow \infty}
  \mu_i^* \eta_i = \phi_U(0)$ uniformly on $U$.
  Pulling back by $\mu_i$, it suffices to construct such
  K\"ahler potentials for the pullback metrics on $U$,
  which we again denote by $g_i$, that are compatible with
  $J_\infty$.

  Construct $u_i(\cdot)$ as in the proof of part (2)
  of the proposition, except
  for the flow $g_i(\cdot)$ instead of $g_\infty(\cdot)$. From (\ref{4.10}), we
  have
  \begin{equation}
    u_i(0) - u(0) = - \int_0^{t^\prime} \log
    \frac{\omega^n(s)}{\omega_i^n(s)} \: ds.
    \end{equation}
  Then
    \begin{equation}
    \|u_i(0) - u(0)\|_{C(U)} \le \int_0^{t^\prime} \left\| \log
    \frac{\omega^n(s)}{\omega_i^n(s)} \right\|_{C(U)} \: ds.
    \end{equation}
    Using (\ref{4.13}) and dominated convergence, it follows that
    $\lim_{i \rightarrow \infty} u_i(0) = u(0)$ uniformly on $U$.

    Recall the functions $v_U$ and $w_U$ constructed in part (2), using the
    $\partial \overline{\partial}$-lemma. Construct functions
    $v_i$ and $w_i$ analogously for the metric $g_i$.
From the smooth convergence
of $\{g_i(t^\prime)\}_{i=1}^\infty$
to $g_\infty(t^\prime)$, and the explicit proof of the
$\partial \overline{\partial}$-lemma
\cite[Lemma I.(3.29) and Proposition III.(1.19)]{Demailly (2012)},
we can assume that
$\{v_{i}\}_{i=1}^\infty$ converges smoothly to $v_{\infty}$, and
$\{w_{i}\}_{i=1}^\infty$ converges smoothly to $w_{\infty}$.
Put
\begin{equation} \label{5.15}
\phi_{i}(0) = v_{i} + t^\prime w_{i} + u_i(0).
\end{equation}
By construction, $\phi_i(0)$ is a K\"ahler potential for $\omega_i$
on $U$ or,
more precisely, for $\mu_i^* \omega_i$.
We have shown that $\lim_{i \rightarrow \infty} \phi_i(0) = \phi_U(0)$ 
uniformly on $U$. Finally, for large $i$, put
$\eta_i = (\mu_i^{-1})^* \phi_i(0)$. Then $\eta_i$ is a 
smooth local K\"ahler potential for $g_i$ on $\mu_i(U)$.

We momentarily exclude the case when $K > 0$ and
$\diam(X_\infty, d_\infty) = \frac{\pi}{\sqrt{2K}}$.
We know that $\eta_i - d_{K, p_i}^2/2$ is plurisubharmonic.
As
\begin{equation}
\lim_{i \rightarrow \infty} \mu_i^*
\left( \eta_i - d_{K, p_i}^2/2 \right) =
\phi_U(0) - d_{K, p}^2/2
\end{equation}
uniformly on $U$,
it follows that
$\phi_U(0)  - d_{K, p}^2/2$ is plurisubharmonic on $U$.

If $K > 0$ and $\diam(X_\infty, d_\infty) = \frac{\pi}{\sqrt{2K}}$ then we
use the fact that $BK \ge \lambda^2 K$ for $\lambda \in (0,1)$, and
$\diam(X_\infty, d_\infty) < \frac{\pi}{\lambda \sqrt{2K}}$, so
$\phi_U(0)  - d_{\lambda^2 K, p}^2/2$ is plurisubharmonic on $U$.
We
take the limit as $\lambda \rightarrow 1$, as in the proof of
Proposition \ref{3.6}, to again conclude that
$\phi_U(0)  - d_{K, p}^2/2$ is plurisubharmonic on $U$.

Given the holomorphic disk $\Sigma \in X_\infty$.
we know that the restriction of
$\phi_U(0) - d_{K, p}^2/2$ to $\Sigma \cap U$ is subharmonic.
Hence
\begin{equation} \label{4.27}
\sqrt{-1} \partial \overline{\partial} d_{K, p}^2 \Big|_{\Sigma \cap U}/2 \le
\sqrt{-1} \partial \overline{\partial} \phi_U(0) \Big|_{\Sigma \cap U} =
     {\mu_\infty}\Big|_{\Sigma \cap U}.
\end{equation}
Then
\begin{equation} \label{4.28}
\sqrt{-1} \partial \overline{\partial} d_{K, p}^2 \Big|_{\Sigma}/2 \le
     {\mu_\infty}
\end{equation}
globally, as measures on $\Sigma$.

Given $\epsilon \in \left(0, \frac{1}{10} \right)$,
define $f_\epsilon : {D^2} \rightarrow \R$ by
\begin{equation} \label{4.29}
  f_\epsilon \left( re^{i \theta} \right) =
  \begin{cases}
    \log(\epsilon) + \epsilon &
    \mbox{ if } 0 \le r \le \epsilon, \\
        \log(r) + \epsilon &
        \mbox{ if } \epsilon \le r \le e^{- \epsilon}, \\
        0 &
    \mbox{ if } e^{- \epsilon} \le r < 1.
  \end{cases}
\end{equation}
Then $\log(|z|) \le f_\epsilon(z) \le 0$, and
$\sqrt{-1} \partial \overline{\partial} f_\epsilon$ exists as a measure.
We have
\begin{align} \label{4.30}
  \int_{\Sigma} (\sqrt{-1} \partial \overline{\partial} f_\epsilon)
  d_{K,p}^2 = & \:
\frac12 \int_0^{2 \pi} \int_0^1 (\partial_r (r \partial_r f_\epsilon)) \:
d_{K,p}^2(r,\theta) \: dr \: d\theta \\
= & \: \frac12 \int_0^{2 \pi} \int_0^1
\left( \delta_{\epsilon}(r) - \delta_{e^{- \epsilon}}(r) \right)
d_{K,p}^2(r,\theta) \: dr \: d\theta \notag \\
= & \: \frac12 \int_0^{2 \pi} 
\left( d_{K,p}^2(\epsilon,\theta) -
d_{K,p}^2 \left(e^{- \epsilon},\theta \right) \right)
d\theta. \notag
\end{align}

Let $\widehat{f}_\epsilon \in C^\infty_c(D^2)$ be a smooth
nonpositive approximation to
$f_\epsilon$, obtained by rounding out the corners at $r = \epsilon$ and 
$r = e^{- \epsilon}$. Since $\widehat{f}_\epsilon$ is nonpositive,
equation (\ref{4.28}) gives
\begin{equation} \label{4.31}
  \frac12 \int_{\Sigma} \widehat{f}_\epsilon \cdot
  \sqrt{-1} \partial \overline{\partial} d^2_{K,p} \ge
  \int_{\Sigma} \widehat{f}_\epsilon d\mu_\infty.
\end{equation}
Passing to a limit as $\widehat{f}_\epsilon$ approaches
${f}_\epsilon$, it follows from (\ref{4.30}) that
\begin{equation} \label{4.32}
\frac14 \int_0^{2 \pi} 
\left( d_{K,p}^2(\epsilon,\theta) -
d_{K,p}^2 \left(e^{- \epsilon},\theta \right) \right)
d\theta \ge
\int_{\Sigma} {f}_\epsilon d\mu_\infty \ge
\int_{\Sigma} \log|z| \: d\mu_\infty.
\end{equation}
Taking the limit as $\epsilon \rightarrow 0$ gives
\begin{equation} \label{4.33}
\frac{\pi}{2} d_{K,p}^2(0)
-  \frac14 \int_0^{2 \pi} 
d_{K,p}^2 \left( e^{i \theta} \right)
d\theta \ge
\int_{\Sigma} \log|z| \: d\mu_\infty,
\end{equation}
or
\begin{equation} \label{4.34}
d_{K,p}^2(0) \ge 
\frac{1}{2\pi} \int_0^{2 \pi} 
d_{K,p}^2 \left( e^{i \theta} \right)
d\theta +
\frac{2}{\pi} \int_{\Sigma} \log|z| \: d\mu_\infty.
\end{equation}
This proves the proposition.
\end{proof}

\begin{remark} \label{4.35}
  In the collapsing case, i.e. if $\lim_{i \rightarrow \infty}
  \vol(B(p_i, 1)) = 0$, there is no direct analog of Proposition
  \ref{4.1} since the limit space need not be K\"ahler, even if it is
  smooth.  For example, a sequence of flat $2$-tori can converge in the
  Gromov-Hausdorff sense to a circle. 

  If there are uniform two-sided sectional curvature bounds then one
  can take a limit in the sense of \'etale groupoids
  \cite[Section 5]{Lott (2007)}, even in the collapsing case.
  The conclusion is that there is a $W^{2,p}$-regular K\"ahler metric on
  the unit space of the groupoid, with $BK \ge K$.

  Natural examples in which there is collapsing with a K\"ahler limit space
  arise in the long-time behavior of the K\"ahler-Ricci flow.
\end{remark}

As a consequence of Proposition \ref{4.1}, we see that if a noncollapsed
pointed Gromov-Hausdorff limit of a sequence of
K\"ahler manifolds happens to be a smooth Riemannian manifold,
and if the K\"ahler manifolds in the sequence have $BK \ge K$, then the
limit is a K\"ahler manifold with $BK \ge K$.

\begin{corollary} \label{4.36}
  Let $\{ (M_i, p_i, g_i) \}_{i=1}^\infty$ be a sequence of pointed
  $n$-dimensional complete K\"ahler manifolds with $BK \ge K$, that
  converges in the pointed Gromov-Hausdorff topology to a
  smooth pointed $n$-dimensional
  Riemannian manifold $(M_\infty, p_\infty, g_\infty)$. Then
  $(M_\infty, g_\infty)$ is a K\"ahler manifold with $BK \ge K$.
  \end{corollary}
\begin{proof}
  This follows from Propositions \ref{3.15} and \ref{4.1}.
  \end{proof}

As an example of what the limits in Proposition \ref{4.1} look like,
consider the case of two real dimensions.
A smooth oriented surface with a Riemannian metric is also a K\"ahler
manifold.
A lower bound on the sectional curvature is equivalent to
a lower bound on the
holomorphic bisectional curvature. Hence one would expect that
oriented surfaces with lower curvature bounds, in the Alexandrov sense,
could also be limits in the sense of Proposition \ref{4.1}.

\begin{proposition} \label{4.37}
  Let $(X,d)$ be a compact two dimensional length space, with Alexandrov
  curvature bounded below by $2K$. It follows that $X$ is a topological
  manifold; assume that it is oriented. Then $X$ satisfies the conclusions
  of Proposition \ref{4.1}.
  \end{proposition}
\begin{proof}
  One knows that $X$ acquires a conformal structure
  \cite[Theorem 7.1.2]{Reshetnyak (1993)}.
  From \cite{Richard (2018)}, there is a smooth Ricci flow $g(\cdot)$
  on $X \times (0, T]$, preserving the conformal structure,
so that the sectional curvature of $g(t)$ is
bounded below by $2K$, and $\lim_{t \rightarrow 0}
(X, g(t)) = (X, d)$ in the Gromov-Hausdorff topology.
Hence the proof of Proposition \ref{4.1} applies.
\end{proof}

\begin{remark}
  The examples in Proposition \ref{4.37} show the sharpness of the
  regularity estimates in Proposition \ref{4.1}. Consider a conical
  metric on $\R^2$ given by $ds^2 = r^{- 2 \alpha} (dr^2 + r^2 d\theta^2)$,
  with $\alpha \in (0,1)$. A K\"ahler potential is
    $\phi = \const r^{2 - 2 \alpha}$, which is only H\"older-continuous with
    respect to the standard metric on $\R^2$. On the other hand,
    the distance function from the origin is
    $d_0 = \const r^{1 - \alpha}$, so $\phi$ is Lipschitz-regular
    with respect to $d$.
  \end{remark}

\section{Singular spaces with lower bounds on holomorphic bisectional
curvature} \label{sect5}

In Section \ref{sect4} the underlying topological spaces were manifolds,
both in the noncollapsing sequences and in the limit spaces.   In analogy
with Alexandrov geometry, it is natural to ask if there is a notion
for singular spaces of
a lower bound on the holomorphic bisectional curvature.

\subsection{Metric K\"ahler spaces}

In the proof of Proposition \ref{4.1}, an important role was played by
local K\"ahler potentials.  This fits well with the notion of
K\"ahler spaces, which are defined using local potentials on 
possibly singular complex spaces.   

Let $X$ be a reduced complex space of pure dimension $n$
\cite[Chapter 2.5]{Demailly (2012)}.
For each $x \in X$, there is
a neighborhood $U_x$ of $x$ and an embedding $e_x : U_x
\rightarrow \C^{N_x}$ so that
$e(U_x)$ is the zero set of a
finite number of analytic functions defined on
an open set $V_x \subset \C^{N_x}$.

If $X_1$ and $X_2$ are complex spaces then a map $F : X_1 \rightarrow X_2$
is holomorphic if for each $x \in X_1$,  there are such
$U_x$ and $U_{F(x)}$, with $F(U_x) \subset U_{F(x)}$, so that the composite map
$e_{F(x)} \circ F \Big|_{U_x} : U_x \rightarrow \C^{N_{F(x)}}$
equals $\widehat{F} \circ e_x$, where $\widehat{F} : V_x
\rightarrow \C^{N_{F(x)}}$ is holomorphic
\cite[Section 1.3]{Greuel-Lossen-Shustin (2007)}.

A function $\phi$ on $U_x$ is plurisubharmonic if it is the pullback under
$e_x$ of a plurisubharmonic function on $V_x \subset \C^{N_x}$. A
pluriharmonic function on $U_x$ is defined similarly. If $X$ is normal and
$\phi \in C(U_x)$ is plurisubharmonic on $U_x \cap X_{reg}$ then it is
plurisubharmonic on $U_x$ \cite{Fornaess-Narasimhan (1980)}. 

As in \cite{Eyssidieux-Guedj-Zeriahi (2009),Moishezon (1975)},
a (semi)-{\em K\"ahler space} consists of a complex space with a covering
$\{U_{j} \}_{j=1}^\infty$ by such open sets, along with
continuous plurisubharmonic
functions $\phi_j$ on $U_{j}$, so
$\phi_j - \phi_{j^\prime}$ is pluriharmonic on each
$U_{j} \cap U_{{j^\prime}} \neq \emptyset$. Two such collections
$\{ (U_j, \phi_j) \}$ and
$\{ (\widehat{U}_k, \widehat{\phi}_k) \}$ are equivalent if
$\phi_j - \widehat{\phi}_k$ is pluriharmonic on each
$U_j \cap \widehat{U}_k \neq \emptyset$. (In the papers
\cite{Eyssidieux-Guedj-Zeriahi (2009),Moishezon (1975)}
the functions $\phi_j$ are taken to be
smooth and strictly plurisubharmonic, but there is clearly some
flexibility in the definitions.)

We wish to define a metric K\"ahler space, meaning a K\"ahler space with a
metric $d$. Naturally, we want some compatiblity between the K\"ahler
space structure and the metric structure. If the K\"ahler potentials are
smooth then there is a corresponding Riemannian metric and one can require
that $d$ be the corresponding length metric.  If the K\"ahler potentials
are only continuous then it is not clear how to construct a length metric;
see, however, \cite[Theorem 1.3]{Li (2017)}.

An indication of a reasonable compatibility condition for us comes from
the use of $dA$ in (\ref{3.16}). In the smooth setting
$dA$ is both the restriction of the K\"ahler form
to a holomorphic disk, and its two dimensional Hausdorff measure.
Again in the smooth setting, the complex structure and the two
dimensional Hausdorff measure determine the K\"ahler form
and the Riemannian metric. Based on this, we make the following
definition.

\begin{definition} \label{5.1}
  A metric K\"ahler space is a K\"ahler space $X$ equipped with a metric
  $d$ that induces the topology of the complex space $X$,
  so that if $\Sigma$ is an embedded holomorphic disk then for all $j$,
  $\sqrt{-1} \partial \overline{\partial} \phi_j \Big|_{\Sigma}$
  equals the two dimensional
  Hausdorff measure on each $\Sigma \cap U_{j} \neq \emptyset$.
  \end{definition}

We now define a notion of ``$BK \ge K$'' for metric K\"ahler spaces,
which we put in quotes in order to distinguish it from the
condition $BK \ge K$ for smooth K\"ahler manifolds.

\begin{definition} \label{5.2}
  A metric K\"ahler space $X$ has ``$BK \ge K$'' if for every $p \in X$
  and every $j$,
  $\phi_j - d_{K,p}^2/2$ is plurisubharmonic on $U_j$.
  \end{definition}

If $S$ is a subset of $X$ and $d_S$ denotes the distance to $S$ then
we define $d_{K,S}$ in terms of $d_S$ as in (\ref{1.1}). The next lemma
will be used in Section \ref{sect6}.

\begin{lemma} \label{5.3}
  If $X$ has ``$BK \ge K$'' then for any $S \subset X$,
  the function $\phi_j - d_{K,S}^2/2$ is plurisubharmonic on $U_j$.
  \end{lemma}
\begin{proof}
  As $d_S = \inf_{p \in S} d_p$, it follows that
  $d_{K,S} = \inf_{p \in S} d_{K,p}$ and
  $\phi_j - d_{K,S}^2/2 = \sup_{p \in S} (\phi_j - d_{K,p}^2/2)$. Now the
  supremum of a family of plurisubharmonic functions, when upper
  semicontinuous, is also plurisubharmonic
  \cite[Chapter 1, Theorem 5.7]{Demailly (2012)}. As
  $\phi_j - d_{K,S}^2/2$ is continuous, it is hence plurisubharmonic.
\end{proof}

We now show the essential equivalence between 
``$BK \ge K$'' and (\ref{3.16}).

\begin{proposition} \label{5.4}
  If $X$ has ``$BK \ge K$'' then for all embedded
  holomorphic disks $\phi$ in $X$,
  equation (\ref{3.16}) holds. If $X$ is normal 
  then the converse is true.
\end{proposition}
\begin{proof}
  If $X$ has ``$BK \ge K$'' then by
  \cite[Theorem 5.3.1]{Fornaess-Narasimhan (1980)},
  $\phi_j - d_{K,p}^2/2$ is subharmonic on $U_j \cap \Sigma$. Hence
  $\sqrt{-1} \partial \overline{\partial} d_{K,p}\Big|_{\Sigma}^2/2 \le dA$
  globally on $\Sigma$.  As in the proof of
  Proposition \ref{4.1}(3), it follows
  that (\ref{3.16}) holds.

  Suppose that $X$ is normal and (\ref{3.16}) holds. Taking embedded holomorphic
  disks $\Sigma$ in $U_j \cap X_{reg}$, it follows that
  $\phi_j - d_{K,p}^2/2$ is plurisubharmonic on $U_j \cap X_{reg}$. As
  $\phi_j - d_{K,p}^2/2$ is continuous on $U_j$, it is then
  also plurisubharmonic
  on $U_j$.
  \end{proof}

We show that if a K\"ahler orbifold has $BK \ge K$, in the sense of curvature
tensors, then the underlying length space has ``$BK \ge K$''. 
For a summary of the relevant topology and geometry of orbifolds, we refer to
\cite[Section 2]{Kleiner-Lott (2014)}.

\begin{proposition} \label{5.5}
  If ${\mathcal O}$ is a smooth effective
  K\"ahler orbifold with $BK \ge K$, in terms of the curvature tensor on
  local coverings, then the underlying topological space $|{\mathcal O}|$
  with the length metric
  has ``$BK \ge K$''.
\end{proposition}
\begin{proof}
  Given $x \in |{\mathcal O}|$, let $G_x$ be its local group.
  There is a local model ${(\widehat{U}, G_x)}$ around $x$, where
  $\widehat{U}$ is an open subset of $\C^n$ containing $0$, and
  $G_x$ acts effectively
  by holomorphic isometries on $\widehat{U}$ while fixing $0$.
    Put $U = \widehat{U}/G_x$, a neighborhood of $x$, with projection
    $\pi : \widehat{U} \rightarrow U$. By shrinking $\widehat{U}$ if necessary,
    we can assume that there is a K\"ahler potential $\widehat{\phi}$ on it.
    Averaging $\widehat{\phi}$ over $G_x$, we can
    assume that it is $G_x$-invariant. Then there is a unique $\phi \in C(U)$ with
    $\pi^* \phi = \widehat{\phi}$. This gives $|{\mathcal O}|$ the structure
    of a K\"ahler space. With the natural length space structure
    on $|{\mathcal O}|$, it becomes a metric K\"ahler space.

    The regular subset $|{\mathcal O}|_{reg}$ consists of the points with
    trivial local group. It is convex in the sense that if
    $x_1, x_2 \in |{\mathcal O}|_{reg}$ then any minimizing geodesic in
    $|{\mathcal O}|$ from $x_1$ to $x_2$ lies in
    $|{\mathcal O}|_{reg}$, as follows for example from
    \cite[Corollary of Theorem 1.2(A)]{Petrunin (1998)}. Given $p \in |{\mathcal O}|_{reg}$ and a
    local potential $\phi$ defined on an open set $U$, the convexity and
    the fact that $BK \ge K$ on $|{\mathcal O}|_{reg}$ implies that
    $\phi - d_{K,p}^2/2$ is plurisubharmonic on $U \cap |{\mathcal O}|_{reg}$.
    Since $|{\mathcal O}|$ is a normal complex space
\cite{Cartan (1957)}, it follows that
    $\phi - d_{K,p}^2/2$ is plurisubharmonic on $U$.

For any $p \in |{\mathcal O}|$, we can find a sequence
$\{p_i\}$ in $|{\mathcal O}|_{reg}$ converging to $p$.
As each $\phi - d_{K,p_i}^2/2$ is plurisubharmonic on $U$,
we can pass to the
limit and deduce that
    $\phi - d_{K,p}^2/2$ is plurisubharmonic on $U$. Hence
    $|{\mathcal O}|$ has ``$BK \ge K$''.
  \end{proof}

\begin{remark} \label{5.6}
  Proposition \ref{5.5} shows that quotient singularities can occur as
  singularities of metric K\"ahler spaces with a lower bound on the
  holomorphic bisectional curvature.  We do not know what other
  singularities can occur.
  \end{remark}

\subsection{Complex Gromov-Hausdorff convergence}

We now give a notion of Gromov-Hausdorff convergence that is adapted
to metric K\"ahler spaces.  One's first inclination may be to require the
Gromov-Hausdorff approximants to be holomorphic.  However, requiring this
globally would be too restrictive.  Instead we consider Gromov-Hausdorff
approximants in the usual sense, which in turn can be locally approximated
by holomorphic maps.

\begin{definition} \label{5.7}
  A collection $\{(X_i, p_i, d_i)\}_{i=1}^\infty$ of pointed complete
  metric K\"ahler spaces
  converges to a pointed complete
  metric K\"ahler space $(X_\infty, p_\infty, d_\infty)$
  in the pointed complex Gromov-Hausdorff topology if
  for every $k \in \Z^+$, there is a covering of $B(p_\infty, k)$ by
  bounded open sets $\{U_{\infty, j}\}$ and associated
  plurisubharmonic functions $\{\phi_{\infty, j}\}$
  so that for every
  $\epsilon > 0$, if $i$ is sufficiently large then there are
  \begin{itemize}
  \item A pointed $\epsilon$-Gromov-Hausdorff approximation
    $h_i : B(p_\infty, k) \rightarrow B(p_i, k)$ and
  \item Holomorphic maps $r_{i,j} : U_{\infty, j} \rightarrow M_i$ that are
    $\epsilon$-close to $h_i$ on $U_{\infty, j} \cap B(p_\infty, k)$, so that
    $r_{i,j}(U_{\infty, j})$ is contained in a set $V_{i,j}$ with an
    associated plurisubharmonic function $\phi_{i,j}$, and
  \item $r_{i,j}^* \phi_{i,j}$ is uniformly $\epsilon$-close to
    $\phi_{\infty, j}$.
    \end{itemize}
    \end{definition}

Note that in Definition \ref{5.7},
the limit space can have lower dimension than
the approximants. In using Definition \ref{5.7}, we allow ourselves to
pass to equivalent choices of $\{(V_{i,j}, \phi_{i,j})\}$ on $M_i$.

We now show that the ``$BK \ge K$'' condition is preserved
under complex Gromov-Hausdorff limits.

\begin{proposition} \label{5.8}
  If $\lim_{i \rightarrow \infty} (X_i, p_i, d_i) = (X_\infty, p_\infty,
  d_\infty)$ in the pointed complex Gromov-Hausdorff topology, and each
  $(X_i, d_i)$ has ``$BK \ge K$'', then $(X_\infty, p_\infty)$ has
  ``$BK \ge K$''.
  \end{proposition}
\begin{proof}
  Fix $k$. Given $p \in X_\infty$, let
  $\{m_i\}$ be points that approach it relative to the Gromov-Hausdorff
  convergence. Given
  $U_{\infty, j}$ as in Definition \ref{5.7}, we have
  \begin{equation} \label{5.9}
  \lim_{i \rightarrow \infty} r_{i,j}^* \left( \phi_{i,j} -
  d_{K,m_i}^2/2 \right) =
  \phi_{\infty, j} - d_{K,p}^2/2
  \end{equation}
  in $L^\infty(U_{\infty, j})$. As $r_{i,j}$ is holomorphic, it follows that
  $\phi_{\infty, j} - d_{K,p}^2/2$ is plurisubharmonic.
\end{proof}

Finally, in the setting of Proposition \ref{4.1}, a subsequence
converges in the complex Gromov-Hausdorff sense.

\begin{proposition} \label{5.10}
  Let $\{(M_i, p_i, g_i)\}_{i=1}^\infty$ be a sequence of pointed
$n$-dimensional complete K\"ahler manifolds with
$BK \ge K$.
Suppose that there is some
$v_0 > 0$ so that for all $i$, $\vol(B(p_i, 1)) \ge v_0$. Then a
subsequence converges in the pointed complex Gromov-Hausdorff topology.
  \end{proposition}
\begin{proof}
This follows from the proof of part (3) of Proposition \ref{4.1}.
\end{proof}

\section{Tangent cones} \label{sect6}

In this section, we prove an analog of Remark \ref{2.4}(3).

\subsection{Tangent cones as K\"ahler cones}

We first characterize tangent cones of noncollapsed limit spaces.

\begin{proposition} \label{6.1}
  Let $(X_\infty, p_\infty, d_\infty)$ be a limit space from
   Proposition \ref{4.1}. Let $T_{p_\infty} X_\infty$ be a tangent cone of
  $X_\infty$ at $p_\infty$. Then $T_{p_\infty} X_\infty$ is a
  K\"ahler cone that is biholomorphic to
  $\C^n$, with 
  $r^2/2$ as a K\"ahler potential. It has ``$BK \ge 0$''.
  \end{proposition}
\begin{proof}
  As $X_\infty$ is a noncollapsed limit of Riemannian manifolds with a
  uniform lower Ricci bound, $T_{p_\infty} X_\infty$ is a metric cone
  of the same dimension whose
  link has diameter at most $\pi$
  \cite[Theorem 5.2]{Cheeger-Colding (1997)}.
  After passing to a subsequence,
  we can write $(T_{p_\infty} X_\infty, 0) =
  \lim_{i \rightarrow \infty} (M_i, p_i, \mu_i^2 g_i)$,
  a pointed Gromov-Hausdorff
  limit, where $\lim_{i \rightarrow \infty} \mu_i = \infty$.
  Hence $(T_{p_\infty} X_\infty, 0)$ is
  a noncollapsed pointed limit of manifolds with
  the lower bound on $BK$ going to zero. Proposition \ref{4.1} implies that
  it satisfies (\ref{3.16}) with $K=0$.

    Since a neighborhood of $x_\infty \in X_\infty$ is
  biholomorphic to a ball in $\C^n$,
  and $T_{p_\infty} X_\infty$ is a blowup limit, it makes sense that
  it should be biholomorphic to $\C^n$.
  To show this, we first construct the complex structure
on $T_{p_\infty} X_\infty$,
  using the
  K\"ahler-Ricci flow.
  
  By definition, $(T_{p_\infty} X_\infty, 0) =
  \lim_{k \rightarrow \infty}
  \left( X_\infty, p_\infty, \lambda_k d_\infty \right)$ as a pointed
  Gromov-Hausdorff limit, where
  $\lim_{k \rightarrow \infty} \lambda_k = \infty$.
  Let $g_\infty(\cdot)$ be the K\"ahler-Ricci flow constructed in the
  proof of 
  Proposition \ref{4.1}, with $t \rightarrow 0$ limit given by
  $(X_\infty, d_\infty)$.
The estimates (\ref{4.3})-(\ref{4.5}) are valid for
$g_\infty(\cdot)$.
Define
the parabolically rescaled
Ricci flows $g_{\infty,k}(u) = \lambda_k^2 g_\infty(\lambda_k^{-2} u)$.
  After passing to a
  subsequence of the $k$'s, we can assume that there is a
  pointed Cheeger-Hamilton limit
\begin{equation} \label{6.2}
  (T_{p_\infty} X_\infty, 0, g_{\infty,\infty}(\cdot)) =
  \lim_{k \rightarrow \infty} (X_\infty, p_\infty, g_{\infty,k}(\cdot))
  \end{equation}
on the time interval $(0, \infty)$.
Letting $B(0,l)$ denote the $l$-ball around the vertex $0$ in
$T_{p_\infty} X_\infty$, 
in taking the limit there
  are implicit embeddings $\sigma_{k,l} : B(0, l) \rightarrow
  X_\infty$ for large $k$ so that $g_{\infty, \infty}(\cdot) =
  \lim_{k \rightarrow \infty} \sigma_{k,l}^* g_{\infty, k}(\cdot)$
  on $[l^{-1}, l] \times B(0, l)$. In particular,
  $\sigma_{k,l}$ decreases distances by approximately
  $\lambda_k$, when going from $T_{p_\infty} X_\infty$ to
  $(X_\infty, d_\infty)$.

  As in the proof of Proposition \ref{4.1}, after passing to a
  subsequence, the pullbacks $\sigma_{k,l}^* J_\infty$ converge,
  as $k \rightarrow \infty$, to a complex structure on
  $B(0,l)$ (say relative to the metric
  $g_{\infty, \infty}(1)$). Applying a diagonal argument, we obtain the
  complex structure $J_{\infty, \infty}$ on $T_{p_\infty} X_\infty$.
  
  Let $\{z^a\}_{a=1}^n$ be local complex coordinates around $p_\infty$ for
  $X_\infty$. Note that $\sum_{a=1}^n |z^a|^2$ is strictly plurisubharmonic
  near $p_\infty$. Put $z^a_{k,l} = \sigma_{k,l}^* z^a$, which for large $k$ is
  a function on $B(0,l)$ that is holomorphic relative to
  $\sigma_{k,l}^* J_\infty $ and harmonic relative to
  $\sigma_{k,l}^* g_{\infty,k}(1)$. After a linear transformation,
  we can assume that
  $\int_{B(0,1)} z^a_{k,l} \overline{z^b_{k,l}} \:
  d\mu = 
  \delta_{ab}$, where
  $d\mu$ is the $n$-dimensional Hausdorff measure on
  $T_{p_\infty}
  X_\infty$.

  After passing to a subsequence of $k$'s, there is a limit
  $z^a_{\infty,l} = \lim_{k \rightarrow \infty} z^a_{k,l}$, where
  $\{ z^a_{\infty,l} \}_{a=1}^n$ are holomorphic functions on
  $B(0,l)$ with
  $\int_{B(0,1)} z^a_{\infty,l} \overline{z^b_{\infty,l}}
  \: d\mu = \delta_{ab}$.
  By a diagonal argument,
  we obtain independent holomorphic functions
  $\{ z^a_{\infty} \}_{a=1}^n$ on $T_{p_\infty} X_\infty$.
  Let $F : T_{p_\infty} X_\infty \rightarrow \C^n$ be given by
  $F(q) = \{z^a_\infty(q)\}_{a=1}^n$. One sees by approximation that
  $F$ is a proper holomorphic map of degree one, and the level sets of
  $|F|^2$ are Stein domains. The preimage $F^{-1}(w)$ of a point $w \in \C^n$
  is a compact
  subvariety in $T_{p_\infty} X_\infty$, so by the Stein property it is a
  finite set of points. It now follows from
  \cite[Proposition 14.7 on p. 87]{Grauert-Peternell-Remmert (1994)}
  that $F$ is biholomorphic.
  Proposition \ref{5.4} implies that $T_{p_\infty} X_\infty$ has ``$BK \ge 0$''.
  
To see that $r^2/2$ is a K\"ahler potential, we use an argument similar to
\cite[Section 4]{Liu (2018)}.
Let $(M_i, p_i, g_i)$ be a sequence as in the beginning of the
proof. Put $\widetilde{g}_i = \mu_i^2 g_i$ and
$\widetilde{d}_{p_i} = \mu_i d_{p_i}$.
Given $0 < a < b <
\infty$ and $\epsilon > 0$, by
\cite[Proposition 4.38, Corollary 4.42 and Corollary 4.83]{Cheeger-Colding (1996)} there is a smooth approximate distance-squared function
$\rho_i$ for $(M_i, p_i, \widetilde{g}_i)$, defined on the metric annulus
$\widetilde{d}_{p_i}^{-1}(a,b)$, so that
\begin{align} \label{6.4}
  \| \rho_i - \widetilde{d}^2_{p_i} \|_{L^2}^2 & = o(i^0), \\
  \| \widetilde{\nabla} \rho_i - \widetilde{\nabla} \widetilde{d}^2_{p_i} \|_{L^2}^2 & = o(i^0), \notag \\
  \| \widetilde{\Hess} \rho_i - \frac{1}{n} (\widetilde{\triangle} \rho_i)
  \widetilde{g}_i \|_{L^1} & = o(i^0). \notag
  \end{align}
From \cite[(4.25) and Proposition 4.35]{Cheeger-Colding (1996)}, we also have
\begin{equation} \label{6.5}
    \| \widetilde{\triangle} \rho_i - n \|_{L^1} = o(i^0).
\end{equation}
Hence
\begin{equation} \label{6.6}
  \| \widetilde{\Hess} \rho_i - 
  \widetilde{g}_i \|_{L^1} = o(i^0).
\end{equation}
In particular,
\begin{equation} \label{6.7}
  \| \sqrt{-1} \partial \overline{\partial} \rho_i - 
  \widetilde{\omega}_i \|_{L^1} = o(i^0). 
\end{equation}

From Proposition \ref{5.10}, after passing to a subsequence,
$\lim_{i \rightarrow \infty} (M_i, p_i, \widetilde{g}_i) =
(T_{p_\infty} X_\infty, 0)$ in the pointed complex Gromov-Hausdorff
topology. It follows from (\ref{6.7})
that if $\phi_\infty$ is a local K\"ahler potential
for $T_{p_\infty} X_\infty$, supported away from $0$, then
$\sqrt{-1} \partial \overline{\partial} \left(
\frac{r^2}{2} - \phi_\infty \right) = 0$ as a current.
Hence $\frac{r^2}{2}$ is a
K\"ahler potential for $T_{p_\infty} X_\infty - 0$. 

There is some continuous K\"ahler potential $\phi_0$ defined in a
neighborhood $U_0$ of $0$. Then $\frac{r^2}{2} - \phi_0$ is continuous on
$U_0$ and pluriharmonic on $U_0 - 0$. Thinking of it as a function in
a neighborhood of $0 \in \C^n$, it follows that $\frac{r^2}{2} - \phi_0$
extends to a continuous pluriharmonic function on $U_0$ (which is then
actually smooth).  Hence
$\frac{r^2}{2}$ is a K\"ahler potential on $T_{p_\infty} X_\infty$.
\end{proof}

\subsection{Curvature of the $\C P^{n-1}$ quotient}

We denote the generator of radial rescaling on $T_{p_\infty} X_\infty$  by $r \partial_r$.
From \cite[Proof of Proposition 15]{Liu-Szekelyhidi (2019)},
$r \partial_r$ and $J_{\infty, \infty} (r \partial_r)$
generate one-parameter groups
that are holomorphic on an open dense subset of $\C^n \cong
T_{p_\infty} X_\infty$.  The one parameter group $\{ \sigma_t \}$ generated by
$J_{\infty, \infty} (r \partial_r)$ acts isometrically on
$T_{p_\infty} X_\infty$ and preserves level sets of the distance function
$d_0$ from the vertex $p_\infty$. Following terminology about Sasaki manifolds,
we say that the structure is regular if $\{ \sigma_t \}$ comes from
a free $S^1$-action. Then the quotient of $T_{p_\infty} X_\infty$ by
the group action is a cone over a manifold.

In order to put ourselves in the setting of a regular structure, we
assume that $d_0$ is a radially
homogeneous function on $\C^n \cong T_{p_\infty} X_\infty$.
That is, letting $\zeta : \C^n - 0 \rightarrow \C P^{n-1}$ denote
the quotient map,
we assume that there are a number $\delta > 0$ and
a function $H \in C(\C P^{n-1})$ so that
\begin{equation}
d_0(z) = |z|^{\delta} H(\zeta(z))
\end{equation}
on $\C^n - 0$. (As an example, this is the case for a two dimensional cone.)
Then
\begin{equation} \label{6.3}
r \partial_r =
 \delta^{-1} \left( \sum_{\alpha=1}^n z^\alpha \partial_{z^\alpha} + 
 \sum_{\alpha=1}^n \overline{z}^\alpha \partial_{\overline{z}^\alpha} \right)
\end{equation}
and $\{ \sigma_t \}$ is the Hopf action on the level sets of $d_0$.
The quotient of the link $d_0^{-1}(1) = S^{2n-1}$ by the Hopf action is
$\C P^{n-1}$, with a possibly nonstandard quotient metric 
$d_{\C P^{n-1}}$.

Let $T$ be the tautological complex line bundle over
$\C P^{n-1}$, whose fibers are lines through the origin in $\C^n$.
The complement of the zero section in $T$ is biholomorphic to
$\C^n - 0$. We will also let $\zeta : T \rightarrow \C P^{n-1}$ denote
the projection map from $T$ to the base.
Consider a local holomorphic
trivialization of $T$ and let $w$ be the fiber coordinate,
with $w=0$ corresponding to the vertex $0 \in T_{p_\infty} X_\infty$.
Then $d_0^2 =
h |w|^{2 \delta}$ for some locally defined continuous function $h$ on
the base. We put a K\"ahler space structure on $\C P^{n-1}$ by saying that
$\frac12 \log h$ is a local potential.

\begin{proposition} \label{6.8}
$(\C P^{n-1}, d_{\C P^{n-1}})$ is a metric K\"ahler space with ``$BK \ge 2$''.
  \end{proposition}
\begin{proof}
Let $\pi : S^{2n-1} \rightarrow \C P^{n-1}$ be the quotient map.
Fix $z^\prime \in \C P^{n-1}$ and let $S \subset \C^n$ be the corresponding
complex line.

\begin{lemma} \label{6.9}
  Let $(r, s)$ denote a point in the metric cone $T_{p_\infty} X_\infty$
where $r \ge 0$ and $s \in S^{2n-1}$. 
  Put $z = \pi(s)$. Then
  $d((r,s), S) = r \sin \left( d_{\C P^{n-1}}(z, z^\prime) \right)$.
  \end{lemma}
\begin{proof}
By the definition of the metric cone,
\begin{equation} \label{6.10}
  d((r,s), (r^\prime, s^\prime)) =
  \sqrt{r^2 + (r^\prime)^2 - 2 r r^\prime \cos \left( d_{S^{2n-1}}(s,s^\prime)
  \right)}.
\end{equation}
Minimizing over $r^\prime$ gives
\begin{equation} \label{6.11}
  d((r,s), S) = r \min_{s^\prime \in S \cap S^{2n-1}} 
\sin \left( d_{S^{2n-1}}(s,s^\prime)
  \right).
\end{equation}
As the $S^1$-action is isometric, the lemma follows from the definition of
the quotient metric.
\end{proof}

From Lemma \ref{5.3}, we know that
\begin{equation} \label{6.12}
\phi - d_S^2/2 =
\frac12 r^2 \zeta^* \cos^2 d_{z^\prime}^2
\end{equation}
is
plurisubharmonic on $T_{p_\infty} X_\infty - 0 \cong \C^n - 0$.

Working locally on $\C P^{n-1}$ and putting
\begin{align} \label{6.13}
  Dw & = \delta \: dw + w h^{-1} \partial h, \\
  D\overline{w} & = \delta \: d\overline{w} + \overline{w} h^{-1}
  \overline{\partial} h, \notag \\
  \Omega & = \sqrt{-1} \partial \overline{\partial} \log h, \notag
  \end{align}
one finds
\begin{align} \label{6.14}
  \partial r^2 & = 
|w|^{2 \delta} h  w^{-1} Dw \\  
  \overline{\partial} r^2 & = 
  |w|^{2 \delta} h  \overline{w}^{-1} D\overline{w},
   \notag \\
  \sqrt{-1} \partial \overline{\partial} r^2 & =
  \sqrt{-1} |w|^{2 (\delta-1)} h D w \wedge D\overline{w} +
  |w|^{2\delta} h \Omega \notag
\end{align}
as currents.

To show that $\left( \C P^{n-1}, d_{\C P^{n-1}} \right)$ is a metric
K\"ahler space, it remains to show that if $\Sigma$ is a holomorphic
disk in the domain of $h$ then
$\frac12 \sqrt{-1} \partial \overline{\partial} \log h
\Big|_{\Dom(h) \cap \Sigma}$
equals the
two dimensional Hausdorff measure $dA$ on
$\Dom(h) \cap \Sigma$. Put $\Gamma = \zeta^{-1}(\Sigma)$, a
four dimensional submanifold of $T_{p_\infty} X_\infty - 0$.
Let ${\mathcal H}$ denote the four dimensional Hausdorff measure on $\Gamma$.
As
in the proof of Proposition \ref{6.1} there is a K\"ahler-Ricci flow
whose pointed Gromov-Hausdorff limit as
$t \rightarrow 0$ is $T_{p_\infty} X_\infty$.
Let ${\mathcal H}_t$ be the four dimensional Hausdorff measure on
$\Gamma$ coming from $d_t \Big|_{\Gamma}$. It equals
$\frac12 \left( \sqrt{-1} \partial \overline{\partial} \phi(t) \right)^2$,
where $\phi(t)$ is a local K\"ahler potential for the flow.
Using
\cite[Chapter 3.3]{Demailly (2012)}
and proceeding as in the proof of Proposition \ref{4.1}(2), it follows that
$\lim_{t \rightarrow 0} {\mathcal H}_t =
\frac12 \left( \sqrt{-1} \partial \overline{\partial} r^2/2 \right)^2$.
Also as in the proof of Proposition \ref{4.1}(2), we have
$\lim_{t \rightarrow 0} {\mathcal H}_t = {\mathcal H}$.
Hence
\begin{equation} \label{6.15}
{\mathcal H} = 
\frac12 \left( \sqrt{-1} \partial \overline{\partial} r^2/2 \right)^2 =
\frac14 \sqrt{-1} |w|^{4\delta-2} h^2 D w \wedge D\overline{w} \wedge
\Omega
\end{equation}
as a measure on $\Gamma$.

From (\ref{6.14}), the area form on a preimage of $\zeta$ is
\begin{equation}
\frac12 \sqrt{-1} \delta^2 |w|^{2 (\delta-1)} h dw \wedge d\overline{w}.
\end{equation}
Since the area of a level set of $w$ is proportionate to
$h |w|^{2 \delta}$, doing
a fiberwise integration on $\Gamma$ gives
\begin{equation} \label{6.17}
  \int_{|w| \le 1} {\mathcal H} = 
\left( \int_{B^2} \delta^2 |z|^{4\delta - 2} \cdot \frac12
  \sqrt{-1} dz \wedge d\overline{z} \right) h^2 dA.
\end{equation}
On the other hand, from (\ref{6.15}),
\begin{equation} 
  \int_{|w| \le 1} {\mathcal H} = 
\left( \int_{B^2} \delta^2 |z|^{4\delta - 2} \cdot \frac12
  \sqrt{-1} dz \wedge d\overline{z} \right) \cdot \frac12 h^2 \Omega.
\end{equation}
Thus $dA  = \frac12 \Omega$
on $\Dom(h) \cap \Sigma$. Since $\Omega$ equals
$\sqrt{-1} \partial \overline{\partial} \log h$,
this shows that $\left( \C P^{n-1}, d_{\C P^{n-1}} \right)$ is a metric
K\"ahler space. 

Finally, put $C = \cos d_{z^\prime} \in C(\C P^{n-1})$,
which we will identify with its pullback to $T$, and put
\begin{align} \label{6.18}
  D_C w & = Dw + w C^{-2} \partial C^2, \\
  D_C \overline{w} & = D\overline{w} +
  \overline{w} C^{-2} \overline{\partial} C^2. \notag
\end{align}
One finds
\begin{align} \label{6.19}
  \sqrt{-1} C^{-2} \partial \overline{\partial} (r^2 C^2) = &
  \sqrt{-1} |w|^{2(\delta-1)} h D_C w \wedge D_C \overline{w} \: +
  \\
  &
  |w|^{2 \delta} h \left( \Omega + \sqrt{-1}
  C^{-2} \partial \overline{\partial} C^2
- \sqrt{-1} C^{-4} \partial C^2 \wedge \overline{\partial} C^2 \right), \notag
\end{align}
as equalities of currents.
Hence from (\ref{6.12}), it follows that
\begin{equation} \label{6.20}
  \Omega + \sqrt{-1} C^{-2} \partial \overline{\partial} C^2
- \sqrt{-1} C^{-4} \partial C^2 \wedge \overline{\partial} C^2 \ge 0,
\end{equation}
or
\begin{equation} \label{6.21}
  - \sqrt{-1} \partial \overline{\partial} \log C^2 \le \Omega.
\end{equation}
Equivalently,
$\frac12 \log h - d_{2,z^\prime}^2/2$ is plurisubharmonic,
where $d_{2,z^\prime}^2$ is defined in (\ref{1.1}),
which means that
$\left( \C P^{n-1}, d_{\C P^{n-1}} \right)$ has
``$BK \ge 2$''.
\end{proof}

\end{document}